\pdfoutput=1
\documentclass[11pt]{amsart}
\usepackage[left=3.2cm, right=3.2cm, top=3.2cm, bottom=3.2cm]{geometry}

\usepackage[T2A]{fontenc} %
\usepackage[utf8]{inputenc} %
\usepackage{dirtytalk} %
\usepackage{yfonts} %

\usepackage{enumitem}
\usepackage{needspace}
\usepackage{graphics}
\usepackage{mathabx,epsfig}

\usepackage[
backend=biber,
bibencoding=ascii,
style=alphabetic,
sorting=nyt,
hyperref=true,
giveninits=true,
maxcitenames=99,
maxbibnames=99,
maxalphanames=99,
doi=false,
url=false,
eprint=false
]{biblatex}
\renewbibmacro{in:}{}
\DeclareSourcemap{
  \maps[datatype=bibtex]{
    \map{
      \step[fieldset=issn, null]
    }
  }
}
\renewcommand*{\bibfont}{\footnotesize}

\DeclareLabelalphaTemplate{
  \labelelement{
	\field[final]{shorthand}
	\field{label}
	\field[strwidth=3,strside=left]{labelname}
  }
  \labelelement{
	\field[strwidth=2,strside=right]{year}    
  }
}
\renewcommand*{\bibfont}{\small}

\usepackage{tikz} %
\usepackage{tikz-cd} %
\usetikzlibrary{shapes,arrows,calc,shapes,decorations.pathreplacing}
\tikzstyle{arrow} = [thick,->,>=latex]

\usepackage{mathtools} %

\usepackage[linktocpage]{hyperref} %
\usepackage{amsfonts,amssymb,amsmath,mathrsfs,amsthm,amscd}
\usepackage{subcaption} %
\usepackage{float} %

\theoremstyle{plain}
\newtheorem{theorem}{Theorem}[section]
\newtheorem*{theorem*}{Theorem}
\newtheorem{proposition}[theorem]{Proposition}
\newtheorem{lemma}[theorem]{Lemma}

\newtheorem*{conjecture*}{Conjecture}

\theoremstyle{definition}
\newtheorem{definition}[theorem]{Definition}
\newtheorem{example}[theorem]{Example}

\newcommand{\thistheoremname}{}
\newtheorem*{genericthm}{\thistheoremname}
\newenvironment{any}[1]
  {\renewcommand{\thistheoremname}{#1}
   \begin{genericthm}}
  {\end{genericthm}}

\newcommand{\thistheoremnameNum}{}
\newtheorem{genericthm_num}[theorem]{\thistheoremnameNum}
\newenvironment{numbered_any}[1]
  {\renewcommand{\thistheoremnameNum}{#1}
   \begin{genericthm_num}}
  {\end{genericthm_num}}

\theoremstyle{remark}
\newtheorem*{remark}{Remark}

\numberwithin{equation}{section}

\newcommand{\ov}[1]{\overline{#1}}

\newcommand\x{\mathbf x}
\newcommand\y{\mathbf y}
\renewcommand\P{\overline P}
\newcommand\A{\mathcal A}
\newcommand\G{\Gamma}
\newcommand\bdry{\partial}

\def\C{\mathbb{C}}
\def\R{\mathbb{R}}
\def\Z{\mathbb{Z}}
\def\F{\mathbb{F}}
\def\S{\mathcal{S}} %
\def\k{\mathbf{k}}
\let\emptyset\varnothing

\def\i{\iota}
\def\b{\boldsymbol}

\def\dualSmallBar{{}^{\A}\ov{bar}{}^{\A}}
\def\dualSmallBarRed{{}^{\A}\ov{bar}_{r}{}^{\A}}

\begin{document}

\title[Explicit]{Bordered theory for pillowcase homology}

\author{Artem Kotelskiy}
\address{Department of Mathematics \\ Princeton University}
\curraddr{Department of Mathematics \\ Indiana University}
\urladdr{http://artofkot.github.io/}
\email{artofkot@iu.com}
\begin{abstract}  
	We construct an algebraic version of Lagrangian Floer homology for immersed curves inside the pillowcase. We first associate to the pillowcase an algebra $\A$. Then to an immersed curve $L$ inside the pillowcase we associate an $A_\infty$ module $M(L)$ over $\A$. Then we prove that Lagrangian Floer homology $HF(L,L')$ is isomorphic to a suitable algebraic pairing of modules $M(L)$ and $M(L')$. This extends the pillowcase homology construction --- given a 2-stranded tangle inside a 3-ball, if one obtains an immersed unobstructed Lagrangian inside the pillowcase, one can further associate an $A_\infty$ module to that Lagrangian.
\end{abstract}

\maketitle

\section{Introduction}
	\subsection{Background}
		Since the groundbreaking work of Donaldson, gauge theory became a powerful tool to study low-dimensional topology. Gauge theory provides invariants of diffeomorphism types of 4-manifolds, homological invariants of 3-manifolds, and homological invariants of knots inside 3-manifolds. These invariants are defined using moduli spaces of solutions to certain partial differential equations on manifolds. The following two theories emerged over the years: instanton theory, where one counts solutions to anti-self-dual Yang-Mills equations, and Seiberg-Witten theory, where one counts solutions to Seiberg-Witten monopole equations. The corresponding invariants of 3-manifolds are called instanton Floer homology \cite{Flo-ins}, introduced by Floer, and monopole Floer homology \cite{KM}, introduced by Kronheimer and Mrowka.

		Another way to construct homological invariants for 3-manifolds and knots is to use symplectic geometry. The general strategy for 3-manifolds is as follows: one takes a Heegaard splitting $U_1\cup_{\Sigma_g} U_2 = Y^3$, and associates to it two Lagrangians inside a certain moduli space (which should possess a natural symplectic structure): $L(U_1), L(U_2) \rightarrow (M(\Sigma_g),\omega)$. Then the desired Floer homology is defined via Lagrangian Floer homology $HF(L(U_1),L(U_2))$. The first homological invariant of this type is called Heegaard Floer homology, and was introduced by Ozsváth and Szabó \cite{OS-main-1}, \cite{OS-main-2}. See a recent survey article \cite{Juh} for an introduction and numerous applications of this theory.

		It is interesting that these two methods to obtain Floer homologies run in parallel, connected by Atiyah-Floer type conjectures. The following is the original formulation for instantons. Having a Heegaard splitting $Y^3=U_1 \cup_{\Sigma_g} U_2$, associate to $U_i$ and $\Sigma_g$ their $SU(2)$ representation varieties $R(U_i)$, $R(\Sigma_g)$. One then has maps $R(U_i) \rightarrow R(\Sigma_g)$. It was conjectured in \cite{Atiyah} that instanton Floer homology $I(Y^3)$ should be equal to the Lagrangian Floer homology $HF(R(U_1),R(U_2))$. Spaces $R(U_1),R(U_2),R(\Sigma_g)$ are singular, and thus symplectic instanton Floer homology $HF(R(U_1),R(U_2))$ was not possible to define at that moment. The symplectic side of the isomorphism, as well as the proof of Atiyah-Floer conjecture, are still under development. There are different versions of symplectic instanton Floer homology, which should correspond to different versions of instanton Floer homology. Notably, the corresponding conjecture on the monopole side was proved in \cite{KLT} and \cite{CGH}: monopole Floer homology and Heegaard Floer homology are equal.

		For a thorough introduction to the above Floer-theoretic invariants, connections between them, and their applications see \cite{Manol}, and references therein. See also more recent papers \cite{Caz}, \cite{Hor-1}, \cite{Hor-2} on symplectic instanton Floer homology, and \cite{DF} on Atiyah-Floer conjecture.

		Now we turn our attention to knot invariants. The first Floer-theoretic invariant for knots inside 3-manifolds was knot Floer homology $\widehat{HFK}(Y^3,K)$, introduced by Ozsváth and Szabó in \cite{OS-hfk-0} using symplectic geometry. The special case of knots in a sphere is denoted by $\widehat{HFK}(K)=\widehat{HFK}(S^3,K)$. Some properties and applications of knot Floer homology are the following: $\widehat{HFK}(K)$ categorifies the Alexander polynomial, detects the 3-dimensional genus of a knot and hence detects the unknot, detects fiberedness of a knot in $S^3$, and also provides lower bounds for the 4-ball genus of a knot. See \cite{Manol-hfk} and \cite{OS-hfk-1} for an introduction to this invariant. A gauge-theoretic counterpart of knot Floer homology was constructed in \cite{KM-sut} and \cite{Kut}. 

		On the instanton side, only gauge theoretic constructions of knot invariants are fully developed. Floer in \cite{Flo-ins-knots} and Kronheimer and Mrowka in \cite{KM-sut} constructed a knot invariant called sutured instanton knot homology $KHI(K)$. It has properties similar to knot Floer homology, like detecting the genus of a knot (non-vanishing result). In fact, $KHI(K)$ is conjectured to be isomorphic to knot Floer homology $\widehat{HFK}(K)$ (both of them categorify the Alexander polynomial of a knot). In \cite{KM-knot}, \cite{KM-Kh} Kronheimer and Mrowka constructed another knot invariant called singular instanton knot homology, which is denoted by $I^\natural(K)$. In \cite{KM-Kh} they proved that there is a spectral sequence from Khovanov homology $Kh(K)$ to $I^\natural(K)$, and that $I^\natural(K)$ is, in fact, isomorphic to $KHI(K)$. This, together with the non-vanishing result for $KHI(K)$, proved that Khovanov homology detects the unknot.

		In order to better understand and compute $I^{\natural}(K)$, Hedden, Herald, and Kirk in \cite{HHK1} and \cite{HHK2} developed a certain geometric construction. The outcome of their construction is an $\F_2$-vector space $H_{pil}(K,\S,\pi)$, which is called {\it pillowcase homology}. Despite of the fact that $H_{pil}(K,\S,\pi)$ depends not only on the knot $K$, but also on the choices of a Conway sphere $\S$ and a certain perturbation data $\pi$, the authors in \cite{HHK2} conjectured and provided lots of evidence that $H_{pil}(K,\S,\pi)$ is in fact a knot invariant. They also conjectured that $H_{pil}(K,\S,\pi)$ should be the symplectic side of Atiyah-Floer conjecture for singular instanton knot homology $I^{\natural}(K)$, see \cite[Conjecture 6.5]{HHK2}.

		Our primary motivation was to enhance the construction of pillowcase homology. We do it by associating algebraic invariants not only to knots, but also to tangles. Let us first describe pillowcase homology in more detail.
	
	\subsection{Pillowcase homology}
		We sketch how the pillowcase homology construction works in the first two  columns of Figure~\ref{figure:strategy}.

		\begin{figure}[!ht]
		\centering
		\includegraphics[width=1.2\textwidth]{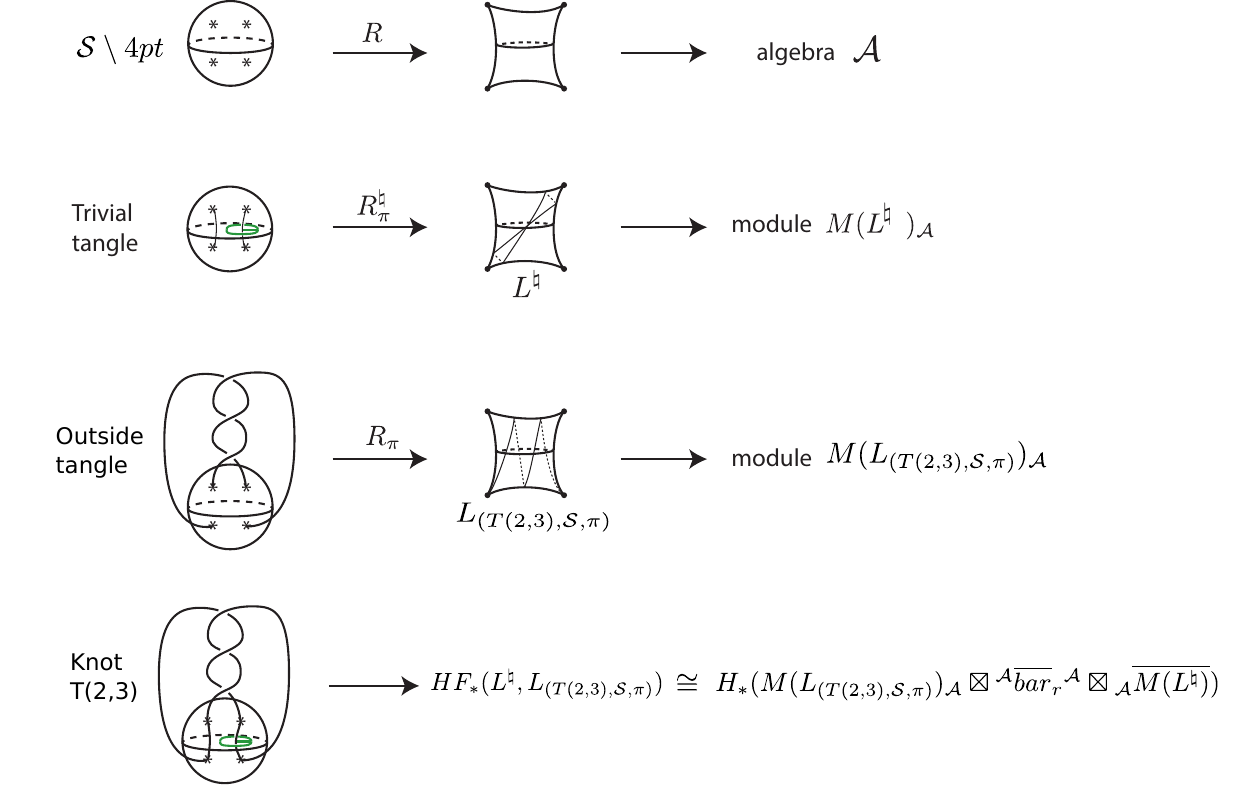}
		\caption{The pillowcase homology construction (the first and the second columns), and its algebraic extension (the 3rd column).}
		\label{figure:strategy}
		\end{figure}

		We now describe the construction. Having a knot in $K \subset S^3$, one first finds a Conway sphere, i.e. a 2-sphere that intersects the knot in four points. We denote it by $\S$, or $(\S,4)$ if we want to emphasize the four points $\S \cap K$. The decomposition of a knot into two tangles by this 2-sphere should be such that one of the tangles is a trivial tangle consisting of two arcs  $A_1,A_2$. Then the steps are as follows:
		\begin{enumerate}[label=(\arabic*)]
			\item To that Conway sphere with four marked points (denote by $\gamma_i, \ i=1,2,3,4,$ loops around those points) one associates a traceless character variety:
			$$R(\S, 4)=\{h \in hom(\pi_1(\S\setminus 4pt),SU(2)) \ | \  tr(h(\gamma_i))=0 \}/conj.$$
			It turns out that this character variety is equal (as an orbifold) to the pillowcase --- a torus quotiented by the hyperelliptic involution $$R(\S, 4)\cong P=S^1\times S^1/ ((\gamma,\theta) \sim (-\gamma,-\theta)),$$ see \cite[Proposition 3.1]{HHK1} for the proof.

			\item To the trivial tangle, which consists of two arcs $A_1,A_2$, one associates an immersed curve $L^{\natural}$ in the pillowcase by the following procedure. First, one adds a circle $H$ with an arc $W$ to the tangle, as shown on the left of the second row of Figure~\ref{figure:strategy}. Then one forms a space of traceless representations:
			\begin{align*}
			R^{\natural}(D^3,A_1 \cup A_2)= \{h \in hom(\pi_1(D^3 \setminus (A_1 \cup A_2 \cup H \cup W)),SU(2)) \ | \\
			| \  tr(h(\mu_{A_i}))=tr(h(\mu_{H}))=0, h(\mu_{W})=-I \  \}/conj.
			\end{align*}

			Because $\S \setminus 4pt \subset D^3 \setminus (A_1 \cup A_2 \cup H \cup W)$, there is a map in the reversed direction $R^{\natural}(D^3,A_1 \cup A_2) \rightarrow R(\S, 4)$. Because this map is singular, and $R^{\natural}(D^3,A_1 \cup A_2)$ is not 1-dimensional, one needs to do a holonomy perturbation of the space. After specifically defined perturbation (see \cite[Section 7]{HHK1}), one gets an immersed circle $L^{\natural}:R^{\natural}_\pi(D^3,A_1 \cup A_2) \looparrowright P$ depicted on the left of Figure~\ref{figure:curve}, missing all four singular points.

			\item To the tangle $ K\setminus (A_1 \cup A_2)$ from the other side one applies almost the same procedure. The only difference is that the circle and the arc $H\cup W$ are not added (this is why here the image will often pass through a singular point). One still needs to perturb $R(D^3, K \setminus (A_1 \cup A_2))$ in this case (see, for example, \cite[Section 11.6]{HHK2} for the case of the (3,4) torus knot). This results in the immersion $L_{(K,\S,\pi)}:R_\pi(D^3, K \setminus (A_1 \cup A_2)) \looparrowright P$. Examples of such immersions for torus knots (with two arcs removed) are depicted in Figure~\ref{figure:curves}.

			\item Having done all that, one associates to the initial knot $K$ a vector space called pillowcase homology. It is equal to the Lagrangian Floer homology $H_{pil}(K,\S,\pi)=HF_*(L^{\natural}, L_{(K,\S,\pi)})$ inside $\P$, where $\P$ is the pillowcase with deleted small neighborhoods of four singular points \footnote{One actually obtains Lagrangians in $P$ and should consider Floer homology where discs do not cross singular points. But one can delete small neighborhoods of singular points to get $\P$, and the corresponding Lagrangian Floer complex will be unchanged.}, see Figure~\ref{figure:pillowcase}.

			\begin{figure}[!ht]
			\centering
			\includegraphics[width=0.4\textwidth]{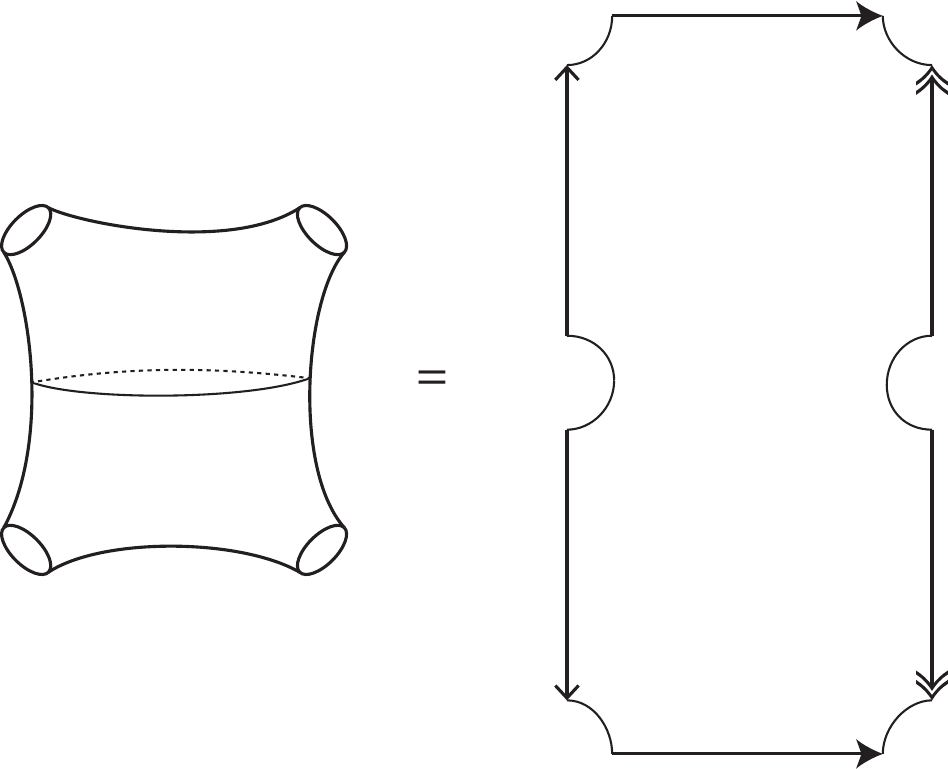}
			\caption{Pillowcase $\P$.}
			\label{figure:pillowcase}
			\end{figure}

		\end{enumerate}

		On the level of chain complexes, the vector space isomorphism $C_{pil}(K,\S,\pi)\cong CI^{\natural}(K)$ is true by construction. In \cite{HHK2} the authors provided lots of examples where the homologies of these chain complexes are indeed isomorphic.
		
		\begin{remark}
		Let us note that it seems hard to prove directly that $H_{pil}(K,\S,\pi)$ is, in fact, a knot invariant. It is unclear how to control the dependence on the Conway sphere $\S$ and the perturbation data $\pi$.

		Also, strictly speaking $H_{pil}(K,\S,\pi)$ was not well defined: apriori it is not clear why the immersion $L_{(K,\S,\pi)}$ is unobstructed and admissible with respect to $L^{\natural}$, in order for Lagrangian Floer homology to be defined without difficulties. This difficulty can be easily resolved, though, as one can homotope $L_{(K,\S,\pi)}$ and make it unobstructed and admissible with respect to $L^{\natural}$.
		\end{remark}

	\subsection{The bordered construction and motivation}
		The construction in this paper is an algebraic enhancement of pillowcase homology. It answers the following question: what algebraic structures one should associate to $L^{\natural}$ and $L_{(K,\S,\pi)}$, in order to be able to recover $H_{pil}(K,\S,\pi)=HF_*(L^{\natural},L_{(K,\S,\pi)})$ algebraically, without looking at the intersection picture on the pillowcase. The relevant objects can be seen in the third column of Figure~\ref{figure:strategy}. Namely, to the Conway sphere $\S$ we have associated a pillowcase $\P$, and now further associate an algebra $\A$. To the trivial 2-stranded tangle we associated the immersed circle $L^{\natural}$ in $\P$, and now further associate a specific module $M(L^{\natural})_\A$, see Figure~\ref{figure:module}. To the tangle from the other side $K \setminus (A_1 \cup A_2)$, similarly already having $L_{(K,\S,\pi)}$, we associate a module $M(L_{(K,\S,\pi)})_\A$. To a union of these two tangles we associate a homology  $H_*(M(L_{(K,\S,\pi)})_{\A} \boxtimes \dualSmallBarRed \boxtimes {}_{\A}\ov{M(L^{\natural})})$. The fact that this algebraic pairing is equal to the pillowcase homology $H_{pil}(K,\S,\pi)$ is the main result of this paper. In the next subsection we formulate a slightly more general result, where we consider any two Lagrangian immersions.

		Let us describe the motivation behind the bordered construction.

		First, it provides a natural candidate for an algebraic invariant of a 2-stranded tangle $T$ inside a ball $D^3$. To such a tangle one can associate an immersed Lagrangian $L(T,\pi):R_\pi(D^3, T) \looparrowright R(\bdry D^3, 4)=P$, and then an $A_\infty$ module $M(L(T,\pi))_\A$ 
		\footnote{Here one must be careful. Definition of $M(L(T,\pi))_\A$ requires a parameterization of the pillowcase $R(\bdry D^3, 4)$. Thus there needs to be additional information, for this parameterization to be fixed. Namely, the boundary of the tangle $(\bdry D^3, 4)$ must be bordered, i.e. parameterized by a standard fixed $(S^2,4)$.}. 
		As with pillowcase homology, there are missing ingredients in this construction: it needs to be proved that the homotopy type of $M(L(T,\pi))_\A$ does not depend on the perturbation $\pi$. 

		Building on this idea, one can isolate the part of $H_{pil}(K,\S,\pi)$ which depends on $L_{(K,\S,\pi)}$, i.e., if one changes $L_{(K,\S,\pi)}$ in some way, it is more natural to understand how $M(L_{(K,\S,\pi)})_\A$ changes, rather than $H_{pil}(K,\S,\pi)=H_*(M(L_{(K,\S,\pi)})_{\A} \boxtimes \dualSmallBarRed \boxtimes {}_{\A}\ov{M(L^{\natural})})$. 

		A very interesting direction of research is to further develop the bordered theory for pillowcase homology $H_{pil}(K,\S,\pi)$ into full bordered theory. Let us briefly describe the way that such a theory would work. The strategy is the following: 
		\begin{enumerate}[label=(\arabic*)]
			\item To understand what algebra should be associated to a $2n$ punctured sphere $(S^{2},2n)$.
			\item To understand what bimodules (over the algebras from the previous step) correspond to tangles inside $S^2 \times I$, which connect $(S^2,2k)$ to $(S^2,2(k+1))$.
			\item To build up a chain complex $C_{alg}(K)$ and prove that its homology $H_{alg}(K)$ is a knot invariant. The construction of $C_{alg}(K)$ should involve composing (via derived tensor product, or morphism space pairing) bimodules from the second step, and modules that correspond to the trivial tangles $M(L_U), M(L^\natural)$ (Examples~\ref{example:triv_tangle},~\ref{example:unknot}).
			\item To prove that this construction, in fact, computes singular instanton knot homology: $H_{alg}(K)\cong I^{\natural}(K)$.
		\end{enumerate}

		This is a difficult project. Even completing the step (1) is hard. The desired algebra should be the algebra of the Fukaya category of the smooth stratum of the representation variety $R(S^2,2n)$. After the pillowcase $R(S^2,4)$, the next space of interest is $R(S^2,6)$. It is already a complicated singular 6-dimensional manifold, see \cite{Kirk}. See also \cite{HK} for the study of $R(S^2,2n)$. Let us note that additional structures on representation spaces
		could help to compute their Fukaya category. For example, in case of Heegaard Floer homology, the Fukaya category of $Sym^g(\Sigma_g \setminus 1pt)$ was computed in \cite{Aur1} using the structure of a Lefschetz fibration over $\C$.

		Nevertheless, if one manages to guess the algebras and bimodules, one can dismiss the underlying geometry and try to prove that the knot invariant is well defined algebraically (step (3)).

		Examples of analogous bordered theories developed for other invariants are: bordered Heegaard Floer homology \cite{LOT-main}, \cite{LOT-bim}; bordered theory for knot Floer homology \cite{OS-hfk-1}, \cite{OS-hfk-2}, \cite{OS-hfk-3}; bordered theories for Khovanov homology \cite{Rob-1}, \cite{Rob-2}, \cite{Manion}. Step (3) for Heegaard Floer homology was done in \cite{Boh}, and for knot Floer homology in \cite{OS-hfk-2}, \cite{OS-hfk-3}.

		Let us mention that there is a related work by Zibrowius \cite{pqMod}: he associates a different algebra to a 4-punctured 2-sphere and modules to 2-stranded tangles, and has a pairing theorem resulting in knot Floer homology.

	\subsection{Main result}
		We construct an algebraic version of Lagrangian Floer homology for two immersed curves inside the pillowcase $\P$. The construction works as follows. To the pillowcase $\P$ we associate an algebra $\A$. To an immersed curve (circle or arc with ends on the boundary) $L$ inside $\P$ we associate an $A_\infty$ module $M(L)_\A$. Then we prove the following pairing result:
		\begin{theorem*}
		Let $L_0, L_1$ be two admissible unobstructed curves in the pillowcase $\P$. Then their Lagrangian Floer complex is homotopy equivalent to the following algebraic pairing of curves: 
		$$CF_*(L_0,L_1) \simeq M(L_1)_{\A} \boxtimes \dualSmallBarRed \boxtimes {}_{\A}\ov{M(L_0)}.$$
		\end{theorem*}

		See Definitions~\ref{def:admissible},~\ref{def:unobstructred} for terms \say{admissible} and \say{unobstructed}; $_{\A}\ov{M(L_0)})$ denotes a dual module;  $\dualSmallBarRed$ is a specific type DD structure, constructed in such a way, that the above homotopy equivalence is true. Another explanation for why we had to work with  $\dualSmallBarRed$  is that, in order to obtain a chain complex, $A_\infty$ modules are usually paired with type D structures, and so it was convenient for us to insert a type DD structure $\dualSmallBarRed$ between the modules. 

		From the above homotopy equivalence it follows that 
		$$HF_*(L_0,L_1) \cong H_*(M(L_1)_{\A} \boxtimes \dualSmallBarRed \boxtimes {}_{\A}\ov{M(L_0)}).$$

		Let us mention that this construction of algebraic Lagrangian Floer homology can be generalized to any oriented surface with boundary $\Sigma$. In order for the process to be analogous, one has to make sure to put enough basepoints on $\bdry \Sigma$ and parameterizing arcs on $\Sigma$, so that the algebra $\A$ becomes directed, i.e. there are no cycles in the generating graph $\Gamma$. Though it is not absolutely necessary --- in \cite{HRW} the authors work out the case of torus with boundary, without requiring the algebra to be directed. There the process is actually reversed: they start with a D-structure (or A-module), and from that they obtain an immersed curve.

		Additionally, the algebraic pairing $H_*(M(L_1)_{\A} \boxtimes \dualSmallBarRed \boxtimes {}_{\A}\ov{M(L_0)})$ gives an algorithm for computing the geometric (i.e. minimal) intersection number of two curves, and the geometric self-intersection number of one curve, on a surface with boundary \footnote{One should treat the case of $Per(L_0,L_1)=\Z$ separately, and subtract 2 from $rk(H_*)$ in order to obtain the geometric intersection number.}.

	\subsection{The underlying structures}\label{sec:reasons}
		The main object behind the scene is the partially wrapped Fukaya category. This special flavor of the Fukaya category was introduced by Auroux in \cite{Aur1}, \cite{Aur2}, in order to reinterpret bordered Heegaard Floer homology via symplectic geometry. From this point of view the main result of this paper can be reformulated as a computation of the enlargement by immersed Lagrangians of the partially wrapped Fukaya category of $\P$. See \cite[Section 3]{Aur3} for the general description of the following process, i.e. what does it mean to generate a category, and what is the Yoneda embedding. 

		Consider the partially wrapped Fukaya category $\mathcal{F}_{pw}(\P)$, where the stops are the basepoints $z_1,z_2,z_3,z_4$ on the left of Figure~\ref{figure:parameterization--algebra}. Note that, because Auroux was considering cohomology instead of homology, we have $CF_*(L_0,L_1)=hom_{\mathcal{F}_{pw}(\P)}(L_1,L_0)$. The parameterization of $\P$ by the red arcs in Section~\ref{sec:param} corresponds to picking a set of Lagrangians $L_1=i_0,\ldots,L_6=j_2 \in \mathcal{F}_{pw}(\P)$. The algebra $\A$, which we define in Section~\ref{sec:algebra}, is the $A_\infty$ algebra $\bigoplus\limits_{i,j} hom_{\mathcal{F}_{pw}(\P)}(L_i,L_j) = \bigoplus\limits_{i,j} CF_*(L_j,L_i)$.

		In Section~\ref{sec:curve->module}, to an immersed curve (circle or arc with ends on the boundary) $L$ inside $\P$ we associate an $A_\infty$ module $M(L)_\A$. It is secretly a module $\bigoplus\limits_{i} hom_{\mathcal{F}_{pw}(\P)}(L,L_i)= \bigoplus\limits_{i} CF_*(L_i,L)$, the image of $L$ under the Yoneda embedding $\mathcal{F}_{pw}(\P) \rightarrow mod_\A$. We do not define it this way, because the partially wrapped Fukaya category was not defined using immersed Lagrangians.

		Then, by \cite[Theorem 1]{Aur2}, we know that $L_1,\ldots,L_6$ generate the category $\mathcal{F}_{pw}(\P)$, which consists of {\it embedded }Lagrangians. This implies that if $L_0,L_1$ are embedded Lagrangians, then we have $HF_*(L_0,L_1)=H_*(hom_{\mathcal{F}_{pw}(\P)}(L_1,L_0)) \cong H_*(Mor_{mod_{\A}}(M(L_0),M(L_1)))$. What we want is to extend this result to $immersed$ Lagrangians, which were not part of the Fukaya category. We also want the algebraic part of the isomorphism to be easily computable, in the light of the morphism spaces of $A_\infty$ modules being often infinitely generated. 

		Instead of extending the notion of the partially wrapped Fukaya category to immersed Lagrangians (although for surfaces this is entirely possible), and then proving that  $L_1,\ldots,L_6$ still generate it, we choose a different method. %
		We first note that the morphism complex can be described in the following way via the bar resolution, see \cite[Proposition 2.10]{LOT-mor}: $Mor_{mod_{\A}}(M(L_0),M(L_1))\cong M(L_1) \boxtimes_\A \ov{Bar_r(\A)} \boxtimes_\A \ov{M(L_0)} $. In Section~\ref{sec:DD}, we describe (following \cite{LOT-mor}) a smaller model for the dual bar resolution, $\dualSmallBar$. Although we do not explicitly prove it, this DD bimodule is homotopy equivalent to $\ov{Bar_r(\A)}$, just as in the case of bordered algebra, see \cite[Proposition 5.13]{LOT-mor}. We then describe explicitly the reduced version of dual small bar resolution $\dualSmallBarRed$.

		Suppose we have two immersed curves $L_0,L_1$ in the pillowcase $\P$. \cite[Theorem 1]{Aur2} suggests that $HF_*(L_0,L_1) = hom^*(L_1,L_0)\cong H_*(Mor(M(L_0),M(L_1)))$. The way we constructed $\dualSmallBarRed$ suggests, that 
		$H_*(Mor(M(L_0),M(L_1))) \cong  H_*(M(L_1)_\A \boxtimes \dualSmallBarRed \boxtimes_\A \ov{M(L_0)})$. We now dismiss the morphism complex, and prove in Section~\ref{sec:thm} directly that $HF_*(L_0,L_1) \cong H_*(M(L_1)_\A \boxtimes \dualSmallBarRed \boxtimes_\A \ov{M(L_0)})$, by interpreting $\dualSmallBarRed$ in a geometric way.

	\begin{any}{Conventions, assumptions, prerequisites}
	We will work over $\F_2$, and we will work with Lagrangian Floer homology, as opposed to cohomology.

	By differential we will mean not only the map $d:C\rightarrow C$, such that $d^2=0$, but also the following. If, for example, $d(x)=y_1+y_2+y_3$, then we say that there is a differential from $x$ to $y_1$ (and from $x$ to $y_2$, and from $x$ to $y_3$). We will denote these differentials by arrows: $x\rightarrow y_1$.

	In this paper we will use $dg$-algebras, $A_\infty$ modules, type DD structures (DD bimodules), and box tensor product $\boxtimes$ operation. For definitions of these objects and operations we refer to \cite{LOT-main} and \cite{LOT-bim}.
	\end{any}

	\begin{any}{Acknowledgments}
	I am thankful to my adviser Zoltán Szabó for suggesting the project, and his continuous support. 
	\end{any}

\Needspace{8\baselineskip}
\section{Immersed curves in the pillowcase: setup} \label{sec:setup}
	\subsection{Pillowcase}
		Fix an oriented torus $T^2=S^1 \times S^1=\R / (2 \pi \cdot \Z^2)$ as a product of two unit circles. The pillowcase is a quotient of the torus by the hyperelliptic involution 
		$$P=T^2/ \ (\gamma,\theta) \sim (-\gamma,-\theta).$$ 
		This quotient has four singular points (which are cones over $\R P^1$), we call them corners. The intersection theory we are interested in happens in the compliment of the corners, or, equivalently, in the compliment of small neighborhoods of the corners. Thus we delete small neighborhoods of the corners and denote the result by 
		$$\P=P \setminus \ U(0,0) \cup U(0,\pi) \cup U(\pi,0) \cup U(\pi,\pi).$$ 
		We will be working with this space from now on (also calling it a pillowcase). Note that it is diffeomorphic to a $2$-sphere with four discs removed.

	\subsection{Immersed curves}
		By a curve $L$ we mean a circle or an arc in the pillowcase: $L:S^1 \looparrowright \P$ or $L:[0,1] \looparrowright \P$. Later we will often write $L$ instead of $Im(L)$ inside $\P$. 

		\begin{definition}\label{def:unobstructred}
		A curve $L$ is called {\it unobstructed} if it is an image of an embedded arc (if $L$ is an arc) or properly embedded line (if $L$ is a circle) in the universal cover of $\P$.
		\end{definition}
		
		If the curve is smooth and properly immersed, the above definition of an unobstructed curve is equivalent (see \cite[Lemma 2.2]{Abouzaid}) to saying that there is no fishtail (see Figure~\ref{figure:fishtail}), and $L$ is not null-homotopic. 
		\begin{figure}[!ht]
			\centering
			\includegraphics[width=0.1\textwidth]{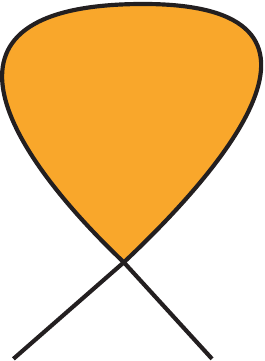}
			\caption{Fishtail.}
			\label{figure:fishtail}
		\end{figure}

		\begin{numbered_any}{Assumption}
		From now on any curve $L$ will be assumed to satisfy the following properties:
		\begin{itemize}
			\item $L$ is smoothly immersed, i.e. the differential is injective. This implies that locally $L$ is an embedding.

			\item If $L$ is a circle, it is contained in the interior $int(\P$). If $L$ is an arc, only the endpoints of it are mapped to the boundary $\partial \P$. Also the endpoints of $L$ should be distinct on $\partial \P$, and transverse to the boundary.

			\item All self-intersections of $L$ are transverse, and there are no triple self-intersections.

			\item $L$ is unobstructed.
		\end{itemize}
		We will call such curves either \say{unobstructed curves}, or simply \say{curves}.
		\end{numbered_any}

		\begin{numbered_any}{Assumption}\label{assum:no_swithches}
		Throughout the paper, when we state different Lagrangian boundary conditions, we will always assume that the maps can be precomposed with the corresponding Lagrangians. More precisely, consider a curve $L$, and consider a map from a surface with boundary into the pillowcase $f:(F,\bdry F)\rightarrow (\P,Im(L))$. We require that this map can be precomposed with $L$ when restricted to the boundary: $f|_{\bdry F}=L \circ g: \bdry F  \rightarrow \P$ for some $g:\bdry F \rightarrow Dom(L)$. We will denote such maps by $f:(F,\bdry F)\rightarrow (\P,L)$.

		An example of a disc $f:(D,\bdry D)\rightarrow (\P,Im(L))$ which does not meet the requirement above is in Figure~\ref{figure:fishtail}.
		\end{numbered_any}
		
		Lagrangian Floer homology is a homology theory for a pair of curves, denoted by $HF(L_0,L_1)$. In order for this to be well defined, we will need pairs to satisfy certain properties.
		\begin{definition}\label{def:admissible}
		We call a pair of curves $(L_0,L_1)$ {\it admissible} if the following properties are satisfied:
		\begin{itemize}
			\item All intersection points in $L_0 \cap L_1$ are transverse. 
			\item There are no triple intersection points.
			\item If both curves $(L_0,L_1)$ are arcs then the following condition must be satisfied. First of all, our pillowcase will be equipped with four basepoints on every boundary component as on the left of Figure~\ref{figure:parameterization--algebra}. Suppose now $a \in \bdry L_0$, $b \in \bdry L_1$, and $a$ and $b$ lie in the same component of $\bdry \P$ with a basepoint $z_i$. Then, with respect to the natural orientation of $\bdry P$, the order of the three points should be first $a$, then $b$, then $z_i$. For example, on the left of Figure~\ref{figure:parameterization--algebra}, the pair $(i_0, i_1)$ is admissible, but the pair $(j_1, j_2)$ is not.
			\item There is no essential immersed annulus with boundary on $L_0$ and $L_1$
			\[A:(S^1\times [0,1],S^1\times\{0\},S^1\times\{1\}) \rightarrow (\P,L_0,L_1).\]
		\end{itemize}
		\end{definition}
		We will not assume that admissibility is always satisfied. Instead we will isotope the curves in order to make them admissible.
\Needspace{8\baselineskip}
\section{Geometric pairing}
	\subsection{Outline}
	Our first goal is to define Lagrangian Floer homology $HF(L_0,L_1)$ for a pair of immersed unobstructed curves $L_0,L_1$ in the pillowcase $\P$. We sketch here the construction, following \cite{HRW}, \cite{Abouzaid}, \cite{SRS}, \cite{HHK2}. The plan is the following:
	\begin{enumerate}[label=(\arabic*)]
		\item For the homology to be well defined we need to restrict the class of curves we consider --- the appropriate class for us are admissible pairs of unobstructed curves (the same setup as in \cite{HRW}), see Definitions~\ref{def:admissible},~\ref{def:unobstructred}. Thus, having two  unobstructed curves $(L_0,L_1)$, we need to know how to isotope $L_0$ to $L'_0$ so that $(L'_0,L_1)$ is admissible.
		
		\item A chain complex $CF(L'_0,L_1)$ is generated over $\F_2$ by intersection points $L'_0 \cap L_1$. The differential $\partial: CF(L'_0,L_1)\rightarrow CF(L'_0,L_1)$ is defined on generators as mod 2 sum
		$$\partial x = \sum_y  \mathcal M(x,y) \cdot y,$$
		where $\mathcal M(x,y)$ counts the number of immersed discs \footnote{Because at the domain of the map we have a disc with two right angles, sometimes these discs are called lunes.} from $x$ to $y$ in the pillowcase, with the right boundary on $L'_0$ and the left boundary on $L_1$, and convex angles at x and y, see Figure~\ref{figure:immersed_disc}. We require these discs to be orientation preserving, and count them up to reparameterizations, which are orientation preserving diffeomorphisms. The main difficulty in this step is to prove that $\mathcal M(x,y)$ is finite for any two generators $x,y$.

		\begin{figure}[!ht]
		\centering
		\includegraphics[width=0.7\textwidth]{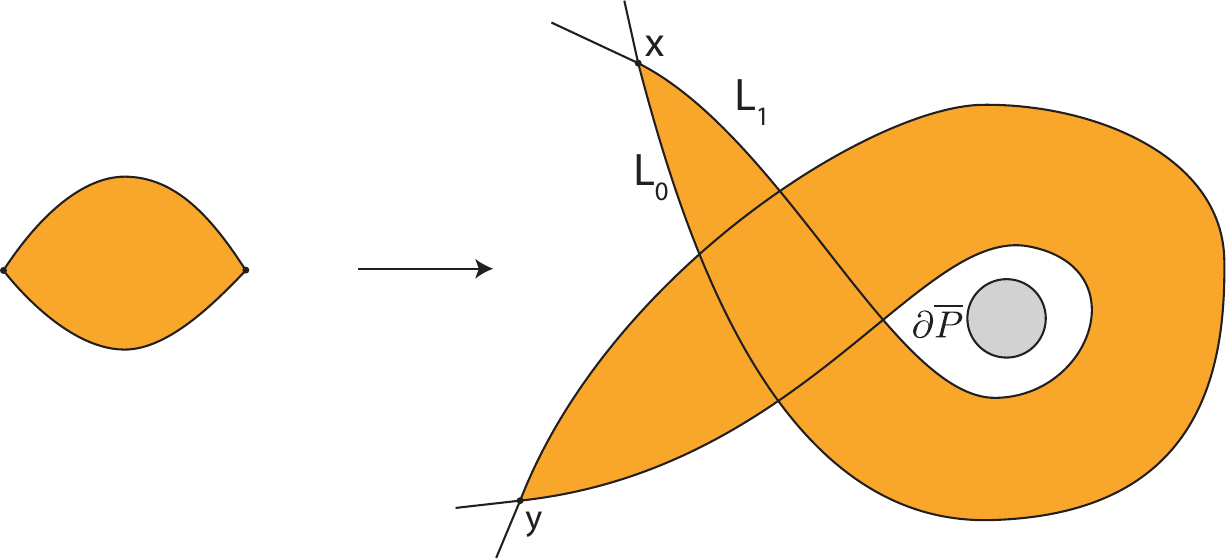}
		\caption{Immersed disc from $x$ to $y$. Note that there are no fishtails due to the presence of $\bdry \P$.}
		\label{figure:immersed_disc}
		\end{figure}

		\item We prove that $\partial^2=0$, and, more generally, that $A_\infty$ relations hold. Then the Lagrangian Floer homology is defined by $HF(L_0,L_1)=H_*(CF(L'_0,L_1),\partial)$. The correctness of the definition follows from the following two statements. 

		\item Suppose $L_0 \sim L_0'$ as basepoint free loops. If $(L_0,L_1)$ and $(L_0',L_1)$ are both admissible pairs, they can be connected through elementary isotopies (of both $L_0$ and $L_1$) called finger moves (see Figure~\ref{figure:finger_move}) such that admissibility does not break down on each step.

		\item If an admissible pair $(L_0,L_1)$ is connected to an admissible pair $(L_0',L_1)$ by a finger move, then $HF(L_0,L_1)=HF(L_0',L_1)$.

		\begin{remark}
		From these two steps it also follows that Lagrangian Floer homology is invariant with respect to isotopies (isotopies of arcs are considered relative to the endpoints).
		\end{remark}

	\end{enumerate}

	\subsection{More details}
	We will follow the plan, outlined in the previous section, giving more attention to the admissibility condition --- the only place where our setup is different from \cite[Sections 2, 3]{HHK2}.
	\begin{enumerate}[label=(\arabic*)]
		\item
			Here we need to show how to isotope $L_0$ to $L_0'$ so that $(L_0',L_1)$ is admissible. The only problematic part is to rule out immersed annuli. Let us first understand when essential annuli exist at all.

			\begin{definition}
			A {\it periodic map} is a smooth annulus $A:(S^1\times [0,1],S^1\times\{0\},S^1\times\{1\}) \rightarrow (\P,L_0,L_1)$. 
			\end{definition}
			Denote by $Per(L_0,L_1)$ the set of homotopy classes of periodic maps.
			\begin{lemma}
			$Per(L_0,L_1)= \Z \ or \ \{0\}$. Admissibility can break down, i.e. essential annuli can exist, only if $Per(L_0,L_1)= \Z$. This is equivalent to $p L_0 \sim q L_1$ as basepoint free loops for some co-prime integers $p$ and $q$ (in particular, both curves should be close immersed circles).
			\end{lemma}
			\begin{proof}
			Denote $L_i:S_i\looparrowright \P$. By Assumption \ref{assum:no_swithches} we can precompose boundaries with $L_i$, i.e. $A:S^1\times\{i\}\rightarrow S_i \xrightarrow{L_i} \P$. Thus, after introducing intersection point between $L_0$ and $L_1$ via isotopy, if necessary, we get the following sequence  
			$$Per(L_0,L_1) \rightarrow%
			\pi_1(S_0 \times S_1) = \pi_1(S_0) \times \pi_1(S_1) \xrightarrow{{L_0}_* + {L_1}_*} \pi_1(\P).$$
			The statement of the lemma follows from 
			$$Per(L_0,L_1)\cong Ker({{L_0}_* + {L_1}_*}).$$
			Surjectivity is straightforward, while injectivity follows from $\pi_2(\P)=0$. 
			\end{proof}

			\begin{definition}
			A {\it shadow} of $A \in Per(L_0,L_1)$ is a two-chain 
			\[Sh(A)=\sum\limits_{open \ D_i \ \subset \ \P \setminus (L_0\cup L_1)} 
			deg(A|_{D_i}) \cdot \ov{D}_i.\]
			\end{definition}
			The key observation is that for $A$ to have immersed orientation preserving representative requires $Sh(A)$ to have all coefficients positive. We call such shadows {\it positive}. In fact, we have:

			\begin{lemma}\label{lemma:non-admiss_iff_positive_shadow_iff_no_intersection}
			The following four statements are equivalent:
			\begin{enumerate}[label=\textbf{\alph*})]
				\item There exists an essential immersed periodic map $A:(S^1\times [0,1],S^1\times\{0\},S^1\times\{1\}) \rightarrow (\P,L_0,L_1)$, and so $(L_0,L_1)$ is not admissible.

				\item There exists a periodic map $A$ with positive shadow for a pair $(L_0,L_1)$ in $\P$.

				\item There exists a periodic map $A$, such that $\tilde{L}_0 \cap \tilde{L}_1 = \emptyset$, where $\tilde{L}_i$ is a lift of $L_i$ to a covering $\tilde{P}$  corresponding to the subgroup $Im(A_*)\subset \pi_1(\P)$.

				\item There exists a periodic map $A$, such that a pair $(\tilde{L}_0,\tilde{L}_1)$ in $\tilde{P}$ has a periodic map with positive shadow.
			\end{enumerate}
			\end{lemma}
			\begin{proof}
			Suppose $Im(A_*)=\langle p L_0\rangle = \langle q L_1\rangle $. Choose a metric so that $\P$ is hyperbolic. Then $\tilde{P}=\mathbb{H}/{\langle \gamma \rangle }$, where $\gamma$ is a translation along a geodesic. This geodesic is the preimage of a geodesic representing $p L_0= q L_1$. Because this translation is fixed point free, it is either parabolic or hyperbolic. This implies that $\tilde{P}$ is homeomorphic to a cylinder. It is now straightforward to see that d) is equivalent to all other statements.
			\end{proof} 
			Now we are prepared to make any pair $(L_0,L_1)$ admissible. Suppose there is an essential immersed periodic map A. Then we isotope one of the curves (say $\tilde{L}_0$) in the covering $\tilde{P}$ to introduce an intersection with another curve, and then push the isotopy down to pillowcase. Note that if $Im(A_*)=\langle p L_0\rangle = \langle q L_1\rangle $, then we need to do isotopies of $\tilde{L}_0$ in $p$ different points, so that it projects down to an isotopy of $L_0$.
		
		\item 
			Here we need to show that $\mathcal M(x,y)$ is finite, assuming $(L_0,L_1)$ is admissible \footnote{In fact, one can prove that  $\mathcal M(x,y)$ if finite for not admissible pairs too, but for our purposes we do not need this.}. Denote by $\pi_2(x,y)$ the space of homotopy classes of smooth discs from $x$ to $y$. We first show that there is a finite number of elements $\phi \in \pi_2(x,y)$, which can have immersed representatives (the relevant condition is shadow $Sh(\phi)$ being positive). Then we show that every such class $\phi$ has exactly one immersed representative from $\mathcal M(x,y)$.

			\begin{lemma}\label{lemma:action}
			In case $\pi_2(x,y) \neq \emptyset$ we have a free and transitive action
			\[
			Per(L_0,L_1)\cong  \pi_2(x,x) \curvearrowright \pi_2(x,y).
			\]
			\end{lemma}
			\begin{proof}
			The general definition of multiplication 
			\begin{align*}
			\pi_2(x,y) \times \pi_2(y,z) &\rightarrow \pi_2(x,z) \\
			(\phi,\psi) &\mapsto \phi*\psi
			\end{align*}
			is given by pinching an arc in the middle of the disc and considering maps $\phi$ and $\psi$ on the resulting two discs (which are connected by one point). The statement follows from this construction.
			\end{proof}

			Thus we get that $\pi_2(x,y)=\{\phi\}, \ \Z \ or \ \emptyset$. Next, let us prove that in the case $\pi_2(x,y)=\Z$ we have only a finite number of elements with immersed representatives.

			The shadow of an element $\phi \in \pi_2(x,y)$ is defined in the same way as for  $A \in Per(L_0,L_1)$.

			\begin{proposition}
			Only a finite number of elements in $\pi_2(x,y)$ have a positive shadow, and thus can have an immersed representative from $\mathcal M(x,y)$.
			\end{proposition}
			\begin{proof}
			Every $0 \neq \phi \in \pi_2(x,x)$ has a shadow with both negative and positive coefficients because $(L_0,L_1)$ is admissible (see Lemma~\ref{lemma:non-admiss_iff_positive_shadow_iff_no_intersection}). For $\psi \in \pi_2(x,x)\cong Per(L_0,L_1)$ and $\phi \in \pi_2(x,y)$ we have $Sh(\psi * \phi)=Sh(\psi) + Sh(\phi)$. This, along with Lemma~\ref{lemma:action}, implies the statement of the proposition.
			\end{proof}

			\begin{proposition}
			The element $\phi \in \pi_2(x,y)$ can have at most one immersed representative, up to smooth reparameterizations.
			\end{proposition}
			\begin{proof}
			This follows from the fact that $\phi \in M(x,y)$ can be reconstructed from its positive shadow, see the proof of \cite[Theorem 6.8]{SRS}, which applies in our case after passing to a universal cover and considering its compact submanifold containing the immersed discs in question.
			\end{proof}
		\item 
			The main idea behind $\bdry^2=0$ is that consecutive immersed discs with convex angles come in pairs (here we use the absence of fishtails), and so they cancel each other. For the details here we refer to \cite[Lemma 2.11]{Abouzaid}.

		\item 
			We have an isotopy $L_0^t$ from admissible $(L_0,L_1)$ to admissible $(L_0',L_1)$. We have problems with keeping this isotopy admissible only if $Per(L_0,L_1)=\Z$. Suppose $A$ is a generator of that group. Then, passing to a covering $\tilde{P}$ corresponding to the subgroup $Im(A_*)\subset \pi_1(\P)$, just like in Lemma~\ref{lemma:non-admiss_iff_positive_shadow_iff_no_intersection}, we can 1) isotope $\tilde{L}_1$ to $\tilde{L}_1'$  in such a way that all isotopies ${\tilde{L}_0}^{t}$ intersect $\tilde{L}_1'$, and $\tilde{L}_0$, $\tilde{L'}_0$ intersect $\tilde{L}_1^t$ at every time 2) do an isotopy $L_0^{t}$ from $L_0$ to $L'_0$ 3) isotopy $\tilde{L}_1'$ back to $\tilde{L}_1$. Both steps 1) and 3) can be done in such a way that the isotopy can be projected to $\P$. This sequence of isotopies keeps the pair admissible all the time.

		\item 
			One possibility here is to use the $A_\infty$ relations to define chain maps between $CF(L_0,L_1)$ and $CF(L'_0,L_1)$, and prove that their composition is homotopic to the identity, see \cite[Lemma 4.2]{HHK2} (with appropriate change of argument because of the weaker notion of admissibility in our case).

			Another approach is to note that the finger move from Figure~\ref{figure:finger_move} on the level of the Lagrangian Floer chain complex corresponds to a cancellation of the differential (see \cite[Appendix C]{SRS}). Here we need to prove that there is exactly one immersed disc between two points on the left of Figure~\ref{figure:finger_move}. There is only one other possibility, which is a disc covering the lower left and the lower right domains on the left of Figure~\ref{figure:finger_move}. But if such an immersed disc exists, we would have an immersed annulus on the right of Figure~\ref{figure:finger_move}, and this would contradict admissibility.
	\end{enumerate}

\Needspace{8\baselineskip}
\section{Pillowcase algebra} \label{sec:pill_algebra}
	In this section we will first parameterize the pillowcase $\P$ by arcs. Then we will associate to this parameterization a $dg$-algebra $\A$.

	\subsection{Parameterization of the pillowcase} \label{sec:param}
		In order to associate a concrete algebra to the pillowcase, we first need to pick a parameterization. The ultimate result --- Theorem~\ref{pairing theorem}, 	algebraic computation of pillowcase homology --- will not depend on the choice of a parameterization, and so we are free to choose a particular, simple enough parameterization of the pillowcase. It is entirely possible to choose a different parameterization, though it might result in a slightly more complicated algebra\footnote{Namely, if the graph $\Gamma$, defined on the next page, has cycles, several proofs in this paper become more complicated, and so we will choose such a parameterization that $\Gamma$ has no cycles.}. The difference between any two parameterizations is measured in a bimodule associated to the mapping class sending one parameterization to another. It turns out that every such mapping class can be decomposed into simple mapping classes called arc-slides, and the bimodules for arc-slides are easy to understand. This was carried out in detail in \cite{LOT-arc}, for the case of genus g surfaces with one boundary component.

		A parameterization consists of basepoints on the boundary $\bdry \P$, and a set of non-intersecting embedded arcs with ends on $\bdry \P$. The following properties should be satisfied: each component of $\bdry \P=S^1\cup S^1\cup S^1\cup S^1$ should get at least one basepoint, and cutting along the arcs should result a set of discs each having exactly one basepoint on the boundary. We pick a parameterization of the pillowcase as on the left of Figure~\ref{figure:parameterization--algebra}. Sometimes we will refer to the parameterizing arcs as the \say{red arcs}.

		\begin{figure}[!ht]
		\centering
		\includegraphics[width=1.1\textwidth]{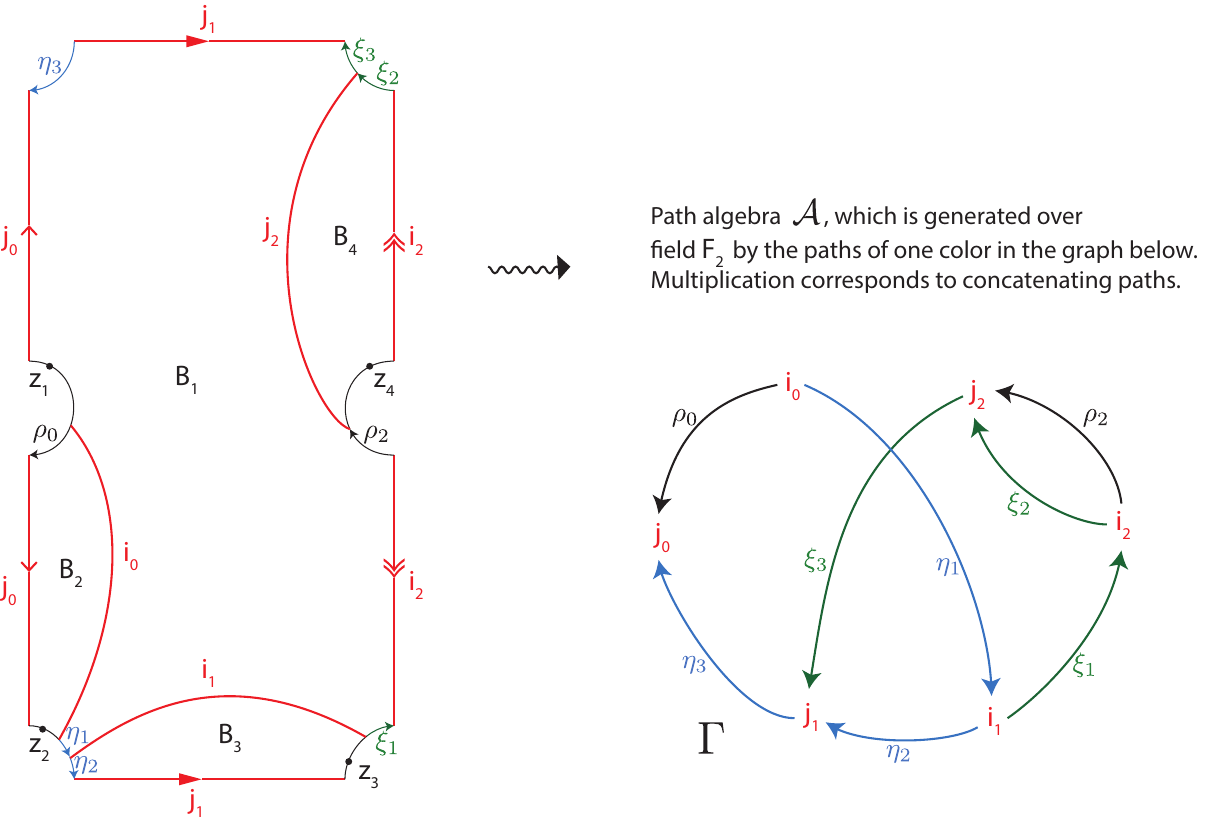}
		\caption{Parameterization of the pillowcase, and the corresponding algebra.}
		\label{figure:parameterization--algebra}
		\end{figure}

	\subsection{Pillowcase algebra} \label{sec:algebra}
		Any parameterization of $\P$ specifies a graph $\G$ --- the vertices are the arcs in the parameterization, and the edges are chords between the arcs on the boundary of $\P$, which do not pass through basepoints. One can see the graph corresponding to our parameterization on the right of Figure~\ref{figure:parameterization--algebra}.

		\begin{definition}
		{\it Pillowcase algebra} $\A$ is a path algebra of the graph $\G$. It means that it is generated over $\F_2$ by paths in $\G$ consisting of edges of one color (or same letters), and concatenating of paths corresponds to multiplication. When concatenating is not possible, or gives a path with edges of different colors, the multiplication results in zero. We mentioned that we want to have a $dg$-algebra corresponding to the pillowcase --- we define the differential to be trivial on $\A$. The subalgebra generated by vertices $\mathcal I=\langle i_0, \ i_1, \ i_2, \ j_0, \ j_1, \ j_2\rangle $ is called {\it idempotent subalgebra}. 
		\end{definition}

		\begin{any}{Explicit description of $\A$} 
		Algebra $\A$ is generated by the following elements (we specify here only those non-trivial multiplications which do not involve vertices):
		\begin{align}\label{description:algebra_elements}
		\A=\langle &i_0, \ i_1, \ i_2, \ j_0, \ j_1, \ j_2, \  \\
			&\rho_0, \ \rho_2, \ \xi_1, \ \xi_2, \ \xi_3, \  \nonumber\\
			&\xi_{12}=\xi_1\xi_2, \ \xi_{23}=\xi_2\xi_3, \ \xi_{123}=\xi_1\xi_2\xi_3=\xi_{12}\xi_3=\xi_1\xi_{23}, \ \nonumber\\ 
			&\eta_1, \ \eta_2, \ \eta_3, \ \eta_{12}=\eta_1\eta_2, \ \eta_{23}=\eta_2\eta_3, \ \eta_{123}=\eta_1\eta_2\eta_3=\eta_1\eta_{23}=\eta_{12}\eta_{3} \rangle _{\F_2} \nonumber.
		\end{align}
		Regarding multiplications which involve vertices: notice that constant paths (i.e. the vertices) are idempotents, and every path in $\A$ has its own left and right idempotent. These idempotents correspond to vertices of the start and the end of the path. All other vertices annihilate the path. For example, for the path $\xi_{12}$ we have $i_1 \xi_{12} j_2=\xi_{12}$, and multiplication by other idempotents results in zero.
		\end{any}
	  
\Needspace{8\baselineskip}
\section{From curves to modules} \label{sec:curve->module}
	To an immersed curve $L$ in $\P$ we associate a right $A_\infty$ module $M(L)_{\A}$ over the algebra $\A$. Before defining the module we need to isotope $L$ appropriately.

	\subsection{Preliminary isotopies of L} \label{sec:isotopies_to_define_modules}
		\begin{definition}
		We call a position of a curve $L$ \emph{nice} if for every parameterizing arc $i$ the pair $(i,L)$ is admissible, and also there are no immersed discs contributing to the differential in $CF(i,L)$.
		\end{definition}

		\begin{lemma}\label{lem:nice_existence}
		Every curve $L$ can be isotoped to be in a $nice$ position.
		\end{lemma}
		\begin{proof}
		First, if $L$ is an arc, we make a perturbation of $L$ in a small neighborhood of $\bdry \P$ by applying a twist along the orientation of $\bdry \P$, such that the end comes close to a basepoint passing all the parameterizing arcs on its way, see Figure~\ref{figure:perturbation}. This ensures that all the parameterizing arcs $i$ are admissible with $L$ as a pair $(i,L)$.

		\begin{figure}[ht]
		\centering
		\includegraphics[width=0.7\textwidth]{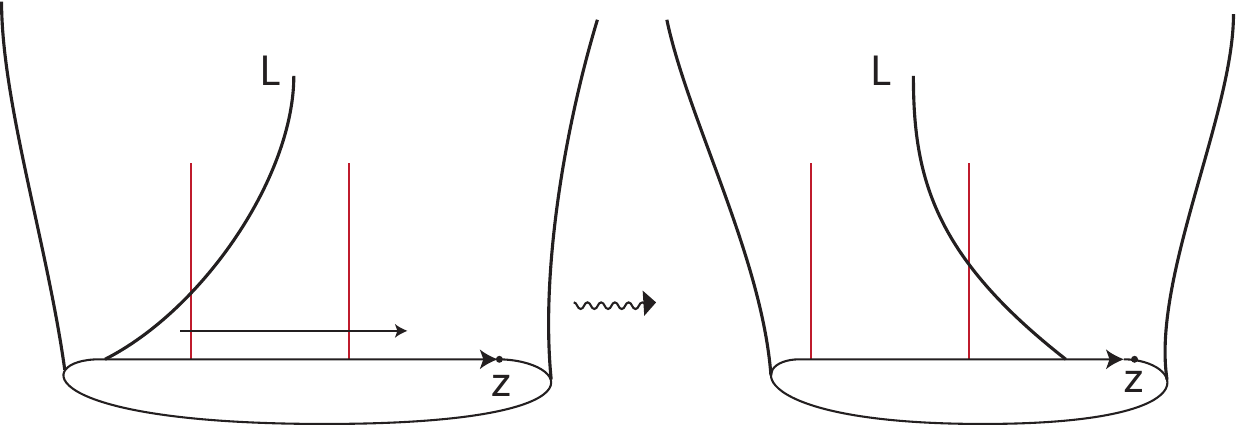}
		\caption{Perturbation near the boundary.}
		\label{figure:perturbation}
		\end{figure}

		Now we exclude the immersed discs contributing to $CF(i,L)$. We first fix the notation of arcs and discs, on which they cut the pillowcase, as in Figure~\ref{figure:curve}. Consider traversing along $L$ on the pillowcase --- this traversing (up to isotopies of $L$ which do not change intersections with arcs) is encoded in the cyclic sequence $S(L)$ of discs $B_k$ (we call them big domains), which are visited by L, as well as the connecting arcs between them. For example for the curve $L^{\natural}$ in Figure~\ref{figure:curve} we have a cyclic sequence $S(L^{\natural})=B_1 j_2 B_4 i_2 B_1 i_0 B_2 j_0 B_1 j_1 B_3 i_1 $. If $L$ is an arc then the sequence is not cyclic.

		We isotope a curve $L$ further, so that the sequence $S(L)$ does not have the same arcs around one big domain, i.e. it does not have a pattern $i B_l i$. Such an isotopy exists because if we have such a pattern, there is a finger move isotopy of $L$ removing $i B_l i$ from the sequence $S(L)$, see Figure~\ref{figure:finger_move}. The length of the sequence decreases, so the process of doing such finger move isotopies has to stop. 

		We claim that once there are not patterns $i B_l i$ in  $S(L)$, there cannot be immersed discs contributing to $CF(i,L)$. Suppose there is an immersed disc; then we can decrease the number of intersections of $L$ with the parameterizing arcs by doing a long finger move along that disc. In the process of doing that finger move every time the number of intersections decreases, it is because we eliminated some pattern $i B_l i$ from $S(L)$. But such patterns did not exist $\implies$ contradiction.

		\begin{figure}[ht]
		\centering
		\includegraphics[width=0.5\textwidth]{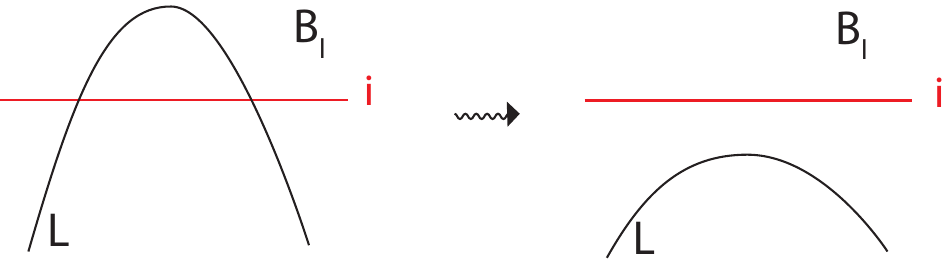}
		\caption{Finger move isotopy.}
		\label{figure:finger_move}
		\end{figure}
		\end{proof}

		\begin{lemma}\label{lem:nice_uniqueness}
		If two isotopic curves $L$ and $L'$ are both in a $nice$ position, their sequences are equal: $S(L)=S(L')$.
		\end{lemma}
		\begin{proof}
		This is straightforward, because both $S(L)$ and $S(L')$ can be reconstructed from the element $[L]=[L'] \in \pi_1(\P,\bdry \P)$.
		\end{proof}

	\subsection{Definition of M(L)}
		\begin{figure}[!h] %
		\includegraphics[width=0.8\textwidth]{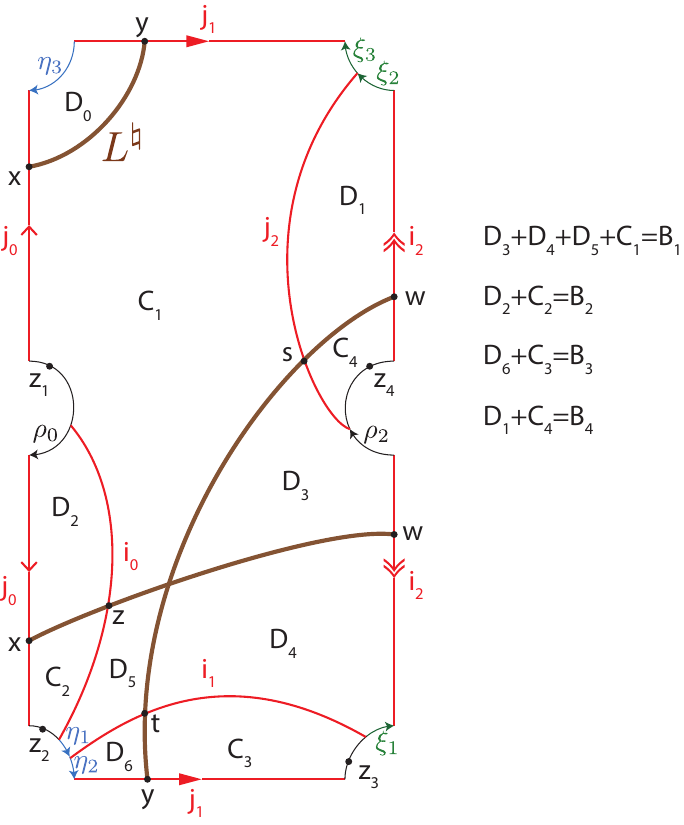}
		\caption{Curve $L^{\natural}$ on the pillowcase.}
		\label{figure:curve}
		\end{figure}

		\begin{figure}[!h] %
		\begin{tikzpicture}[scale=0.4,baseline=1.5cm]
			\foreach \a/\b in {1/w, 2/z, 3/x, 4/y, 5/t, 6/s}{
			\draw (\a*360/6: 4cm) node(\b){\b};
			}
			\draw[arrow] (z) to node[left]{$\rho_0$} (x);
			\draw[arrow] (y) to node[left]{$\eta_3$} (x);
			\draw[arrow] (t) to node[below]{$\eta_2$} (y);
			\draw[arrow] (t) to node[right]{$(\xi_1,\rho_2)$} (s);
			\draw[arrow] (w) to node[right]{$\xi_2$} (s);
			\draw[arrow] (z) to node[above]{$(\eta_1,\xi_1)$} (w);
			\draw[arrow] (z) to node[below]{$(\eta_1,\xi_{12})$} (s);
			\draw[arrow] (t) to node[above]{$\eta_{23}$} (x);
			
		\end{tikzpicture}
		\caption{$M(L^{\natural})_{\A}$, where $L^{\natural}$ is from Figure~\ref{figure:curve}.}
		\label{figure:module}
		\end{figure}
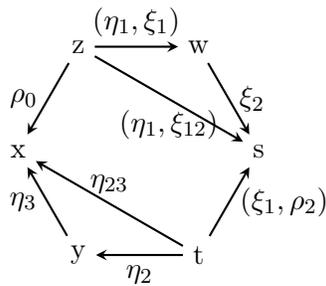

		\begin{figure}[!h]
		\centering
		\includegraphics[width=0.8\textwidth]{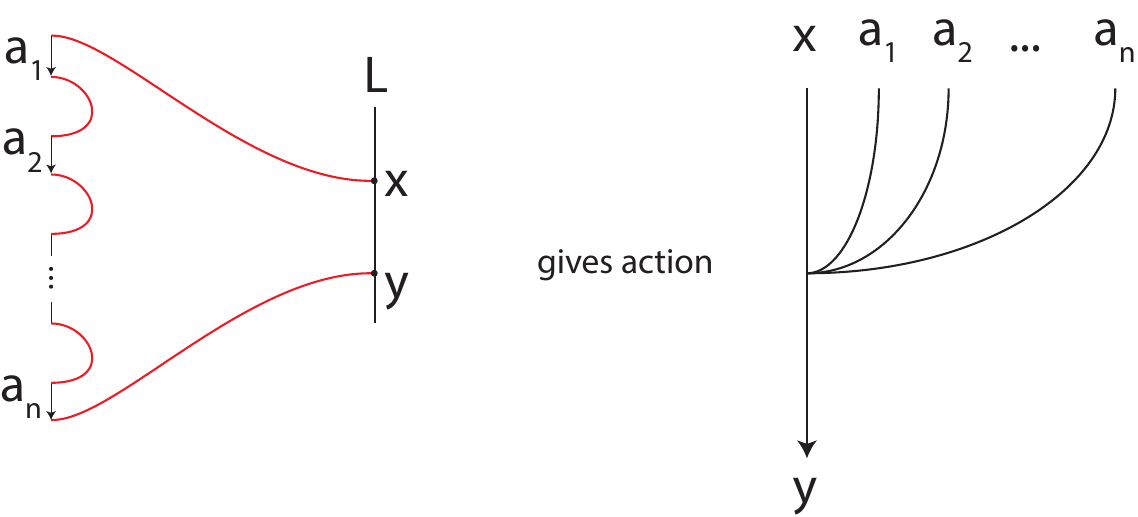}
		\caption{Immersed discs of this type define $A_\infty$ actions of the algebra $\A$ on the module $M(L)$.}
		\label{figure:actions}
		\end{figure}

		By Lemma~\ref{lem:nice_existence} we now assume that $L$ is in a nice position. A right $A_\infty$ module $M(L)_{\A}$ is defined as follows. Over $\F_2$ it is generated by all intersections of the curve $L$ with the red arcs. For example, for the curve $L^{\natural}$ in Figure~\ref{figure:curve} we have $M(L^{\natural})=\langle z,w,s,t,y,x\rangle _{\F_2}$.

		Idempotent subalgebra $\mathcal I$ acts on $M(L)$ from the right: every generator has a unique idempotent which preserves it, this idempotent corresponds to the arc on which this generator is sitting. Other idempotents annihilate the generator. For example for $M(L^{\natural})$ we get the following idempotents for the generators: $z_{i_0},w_{j_0},s_{j_1},t_{i_1},y_{j_2},x_{i_2}$.

		The rest of $\A$ acts on $M(L)$ by counting immersed (in fact they are all embedded, because (1) $L$ is in a nice position; (2) the graph $\Gamma$ does not have cycles) discs, missing basepoints, from one generator to another generator, such that the right boundary of the disc is mapped to arcs or $\bdry \P$, left boundary of the disc is mapped to $L$, and all the angles are convex --- see Figure~\ref{figure:actions}. Non-idempotent elements of the algebra $a_1, \dots, a_n$ which such a disc picks up on $\bdry \P$ give an $A_\infty$ action  $x\otimes_{\mathcal I} a_1 \otimes_{\mathcal I} a_2 \otimes_{\mathcal I} \dots \otimes_{\mathcal I} a_n \rightarrow y$.

		Now, let us explain how $M(L)$ can be recovered combinatorially, by doing it on the example curve $L^{\natural}$ in Figure~\ref{figure:curve}. First, let us count all the \say{basic} discs, which are contained entirely in one of the big domains $B_l$. Because there is one basepoint in each big domain $B_l$, there is exactly one basic disc between two consecutive generators. For the curve $L^{\natural}$ from Figure~\ref{figure:curve} the basic discs are: $z \xrightarrow{D_2} x $, $y \xrightarrow{D_0} x $, $t \xrightarrow{D_6} y $, $t \xrightarrow{D_4 + D_3} s $, $w \xrightarrow{D_1} s $, $z \xrightarrow{D_4 + D_5} w $.  Thus first we get a circle (if $L$ is an arc we get a sequence) of generators and actions between them, see Figure~\ref{figure:module} for an example of the module $M(L^{\natural})$. Notice that except the basic actions that form a circle there are two extra actions: $z \xrightarrow{D_1 + D_4 + D_5} s $, $t \xrightarrow{D_0 + D_6} x$. These are the actions which correspond to discs which are formed by juxtaposing basic discs along the arcs. These discs ensure that $d^2=0$ in our $A_\infty$ module. Note that every immersed disc can be decomposed into basic discs. Also every basic disc is contained in the finite number of discs --- otherwise we would have a cycle of chords on the $\bdry P$, and this is not possible because the graph $\Gamma$ has no cycles.

		The fact that the module $M(L)_{\A}$ is well defined follows from Lemma~\ref{lem:nice_uniqueness}.
\Needspace{8\baselineskip}
\section{Algebraic pairing}
	Here we will describe how to compute $HF(L_0,L_1)$ in terms of $M(L_0)_{\A}$ and $M(L_1)_{\A}$. One can skip most of this section, and jump straight to Definition~\ref{def:dualSmallBarRed}. The material before that definition is included to show how we arrived at that definition.

	\subsection{Koszul dual algebra}
		First, note that our algebra $\A$ is a 1-strand moving algebra $\A(\mathcal Z,1)$ (see \cite[Definition 2.6]{Zarev})  of the arc diagram $\mathcal Z$ drawn in Figure~\ref{figure:arc_diagram}.

		\begin{figure}[!ht]
		\centering
		\includegraphics[width=0.4\textwidth]{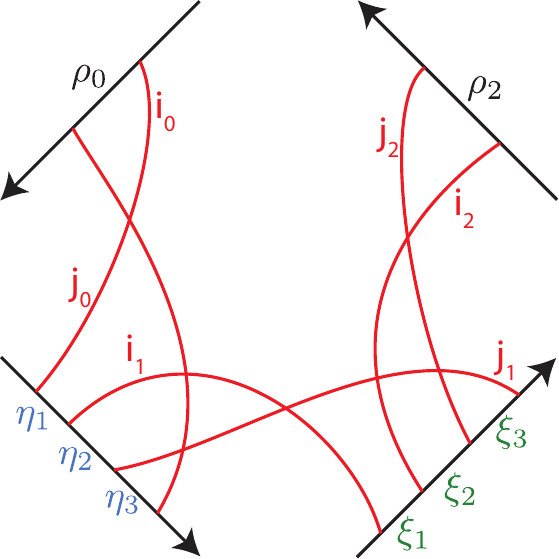}
		\caption{Arc diagram $\mathcal Z$, whose 1-strand moving algebra is $\A$.}
		\label{figure:arc_diagram}
		\end{figure}

		\begin{definition}
		Let us define a new algebra $\A^{5s}$ as $6-1=5$-strand moving algebra $\A(\mathcal Z, 5)$ of the same arc diagram. We call it {\it Koszul dual algebra} to algebra $\A$.
		\end{definition}

		\begin{any}{Explicit description of $\A^{5s}$}
		
		First, consider a new graph $\Gamma'$ in Figure~\ref{figure:dual_graph}, consisting of the reversed paths in the graph $\Gamma$ from Figure~\ref{figure:parameterization--algebra}. Our algebra is a path algebra of graph $\Gamma'$, i.e. 
		\[ \A^{5s}=\langle \{\text{paths in }\Gamma' \}\rangle _{\F_2}. \]
		Notice that now all the paths are of the same color. Let us denote the non-idempotent elements by $_{i'}a(e_1,e_2,\ldots,e_m)_{j'}$, where $e_i$ is an edge in $\Gamma'$, and indices are the start and the end of the path. As before, multiplication corresponds to concatenating paths, and so non-zero multiplications are all of the form 
		$$_{i'}a(e_1,e_2,\ldots,e_m)_{j'} \cdot {}_{j'}a(e_{m+1},e_{m+2},\ldots,e_{m+l})_{k'}=_{i'}a(e_1,e_2,\ldots,e_{m+l})_{k'}.$$
		There is a natural 1-1 correspondence between idempotents of $\A$ and $\A^{5s}$ given by $i \leftrightarrow i'$. Also, edge $e$ in graph $\Gamma'$ naturally gives an element $a_{\A}(e) \in \A$ by \say{reversing} the path, for example, $\xi'_{32} \mapsto \xi_{23}$.

		Differential this time is not zero. First, we specify differential on \say{linear} elements, consisting of one edge (as always, we list only non-zero differentials):
		\begin{align*}
		&d(a(\xi'_{21}))=a(\xi'_{2},\xi'_{1}), \ d(a(\xi'_{32}))=a(\xi'_{3},\xi'_{2}), \  d(a(\xi'_{321}))=a(\xi'_{32},\xi'_{1}) + a(\xi'_{3},\xi'_{21}), \\
		&d(a(\eta'_{21}))=a(\eta'_{2},\eta'_{1}), \ d(a(\eta'_{32}))=a(\eta'_{3},\eta'_{2}), \  d(a(\eta'_{321}))=a(\eta'_{32},\eta'_{1}) + a(\eta'_{3},\eta'_{21}). \\
		\end{align*}
		These induce differential on paths that consist of more edges by Leibniz rule. For example, for 3-edge paths we have 
		\[ d(a(e_1,e_2,e_3))=d(a(e_1) \cdot a(e_2) \cdot a(e_3))=d(a(e_1)) \cdot a(e_2) \cdot a(e_3) + a(e_1) \cdot d(a(e_2)) \cdot a(e_3) + a(e_1) \cdot a(e_2) \cdot d(a(e_3)). \]
		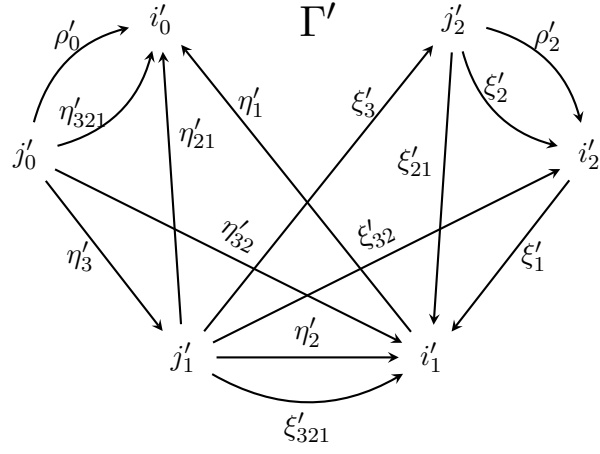
\begin{figure}[!ht]
		\centering
		\begin{tikzpicture}[scale=0.3,baseline=1.5cm]
			\node[circle](i'_0) at (6,15){$i'_0$};
			\node[circle](j'_0) at (0,9){$j'_0$};
			\node[circle](j'_1) at (7,0){$j'_1$};
			\node[circle](i'_1) at (18,0){$i'_1$};
			\node[circle](i'_2) at (25,9){$i'_2$};
			\node[circle](j'_2) at (19,15){$j'_2$};

			\node[circle] at (13,15){\scalebox{1.5}{$\Gamma'$}};

			\draw[arrow] (j'_0) to node[left]{$\eta'_3$} (j'_1);
			\draw[arrow] (j'_1) to node[above]{$\eta'_2$} (i'_1);
			\draw[arrow] (i'_1) to node[right,pos=0.8]{$\eta'_1$} (i'_0);
			\draw[arrow] (j'_0) to node[above, sloped]{$\eta'_{32}$} (i'_1);
			\draw[arrow] (j'_1) to node[right,pos=0.7]{$\eta'_{21}$} (i'_0);
			\draw[arrow] (j'_0) to[bend right] node[left]{$\eta'_{321}$} (i'_0);
			\draw[arrow] (j'_0) to[bend left] node[above]{$\rho'_{0}$} (i'_0);

			\draw[arrow] (j'_2) to[bend left] node[above]{$\rho'_{2}$} (i'_2);
			\draw[arrow] (j'_1) to node[left,pos=0.8]{$\xi'_3$} (j'_2);
			\draw[arrow] (j'_2) to[bend right] node[above]{$\xi'_2$} (i'_2);
			\draw[arrow] (i'_2) to node[right]{$\xi'_1$} (i'_1);
			\draw[arrow] (j'_1) to node[above,sloped]{$\xi'_{32}$} (i'_2);
			\draw[arrow] (j'_2) to node[left, pos=0.4]{$\xi'_{21}$} (i'_1);
			\draw[arrow] (j'_1) to[bend right] node[below]{$\xi'_{321}$} (i'_1);
		\end{tikzpicture}
		\caption{Graph $\Gamma'$, consisting of the reversed paths in the graph $\Gamma$ from Figure~\ref{figure:parameterization--algebra}.}
		\label{figure:dual_graph}
		\end{figure}
		\end{any}
		\begin{remark}
		For clarity we repeat here our notation: elements of algebra $\A$ are denoted by letters as described in Definition \ref{description:algebra_elements}. Some of those letters are edges of the graph $\Gamma$. For elements of the algebra $\A^{5s}$ the notation is $a(e_1,\ldots,e_m)$, where $e_i$ are the edges of the graph $\Gamma'$. Each edge $e$ in graph $\Gamma'$ naturally gives an element $a_{\A}(e) \in \A$ by \say{reversing} the path.

		\end{remark}

		\begin{remark}
		Although we will not use it, let us note that the algebra $\A^{5s}=\A(\mathcal Z,5)$ is Koszul dual to $\A=\A(\mathcal Z,1)$ in the sense of \cite[Definition 8.5]{LOT-mor}. The proof is the same as for Koszul duality of algebras in bordered Heegaard Floer homology, see \cite[Proposition 8.17]{LOT-mor}. Let us describe the rank-1 (over $\mathcal I$) Koszul dualizing bimodule $^{\A}K^{\A^{5s}}$. We define 
		\[
		K=\langle (i_0,i'_0), (i_2,i'_2),(i_2,i'_2),(j_0,j'_0),(j_1,j'_1),(j_1,j'_1) \rangle _{\F_2} \cong \langle \boldsymbol 1\rangle _{\mathcal I},
		\] 
		with differential $\delta^1:K \rightarrow  \A \otimes_{\mathcal I} K \otimes_{\mathcal I} \A^{5s} $ given by
		\[
		\delta^1(k,k')=\sum\limits_{_{s'} e_{k'} \text{ edge in } \Gamma'} {_{k} {a_{\A}(e)} _s \otimes (s,s') \otimes (a(_{s'}e_{k'}))}.
		\]
		\end{remark}

	\subsection{DD bimodule, and the pairing} \label{sec:DD}
		\begin{definition}
		{\it Dual small bar resolution} of algebra $\A$ is a type DD structure $\dualSmallBar$, whose generators correspond to elements of $ \A^{5s}$, i.e. each element
		$_{i'}{a(e_1,e_2,\ldots,e_l)}_{j'} \in \A^{5s}$ gives an element $ {}_i{b(e_1,e_2,\ldots,e_l)}_j \in \ov{bar}$. The type DD structure on $\ov{bar}$ over $\A$ is given by:
		\begin{align*}
		\delta^1:\ov{bar} \rightarrow \A \otimes_{\mathcal I} \ov{bar} \otimes_{\mathcal I}  &\A , \\
		\delta^1(_i{b(e_1,e_2,\ldots,e_l)}_j)=&\sum\limits_{e \in Edges(\Gamma'), start(e)=j'}
		1 \otimes {}_i{b(e_1,e_2,\ldots,e_l,e)}_k \otimes {_k{a}_{\A}(e)_j}
		+\\
		+ &\sum\limits_{e \in Edges(\Gamma'),end(e)=i'} 
		{_i{a}_{\A}(e)_k} \otimes {_k{b(e,e_1,e_2,\ldots,e_l)}_j} \otimes 1 
		+ \\
		+ &\sum\limits_{e_i \in \{e_1,e_2,\ldots,e_l\}} 
		1 \otimes b(e_1,e_2,\ldots,e_{i-1}) \cdot d(b(e_i)) \cdot b(e_{i+1},\ldots,e_l) \otimes 1.
		\end{align*}
		\end{definition}
		For the explicit description of the elements and the differential in $\dualSmallBar$ see Appendix. For convenience let us write here one example of how the differential acts:

		\[
		\begin{tikzpicture}[scale=1.5]
			\node[]() at (-1.5,1){$\delta^1|_{b(\eta_3',\xi_3',\xi_{21}')} \ =$};
			
			\node[](in1) at (0,2){$_{j_0}b(\eta_3',\xi_3',\xi_{21}')_{i_1}$};
			\node[circle,fill,scale=0.3](diff) at (0,1){};
			\node[](out1) at (0,0){$_{j_0}b(\eta_3',\xi_3',\xi_{2}',\xi_{1}')_{i_1}$};
			\draw[arrow] (in1) to (out1);

			\node[]() at (1.2,1){+};

			\node[](in1) at (2.4,2){$_{j_0}b(\eta_3',\xi_3',\xi_{21}')_{i_1}$};
			\node[circle,fill,scale=0.3](diff) at (2.4,1){};
			\node[](out1) at (2.4,0){$_{j_0}b(\eta_3',\xi_3',\xi_{21}',\eta_{1}')_{i_0}$};
			\node[](out2) at (4,0){$_{i_0}{\eta_{1}}_{i_1}$};
			\draw[arrow] (in1) to (out1);
			\draw[arrow] (diff) to[bend left] (out2);
		\end{tikzpicture}
		\]

		We will simplify $\dualSmallBar$ preserving its homotopy type. For that there exists a convenient tool called \say{cancellation}. Suppose there are two generators in a DD bimodule $\dualSmallBar$ satisfying $\delta^1(x)=y + \ldots$, i.e. there is only one action from $x$ to $y$, and it does not have any outgoing algebra elements (an example would be $b(\eta_3',\xi_3',\xi_{21}') \rightarrow 1\otimes b(\eta_3',\xi_3',\xi_{2}',\xi_{1}') \otimes 1$). Then we can cancel these two generators, i.e., first, erase $x$, $y$ and the arrows involving them from the bimodule, and second, add some other arrows between the generators left in the bimodule, guided by a certain cancellation rule. The outcome is a bimodule $^{\A}\ov{bar}{}'^{\A}$ with less generators, and which is homotopy equivalent to the previous one $ ^{\A}\ov{bar}{}'^{\A}$ $\simeq$ $\dualSmallBar$. See \cite[Section 3.1]{Boh} for the details of how cancellation works.

		We want to cancel all possible differentials in $\dualSmallBar$. It is clear that it does not matter which differentials and in which order we cancel, up to homotopy equivalence the end result $\dualSmallBarRed$ will be the same. Moreover, in our case it turns out that even the isomorphism class of $\dualSmallBarRed$ does not depend on the choice and order of cancellations. We could describe the cancellation process and prove the uniqueness of the resulting bimodule $\dualSmallBarRed$, but we decided to skip this unnecessary step, and simply define $\dualSmallBarRed$ explicitly below. Note that below we change the notation: instead of primes we will write minuses, i.e. $b(\xi'_1)$ becomes $b(-\xi_1)$, except for the constant path elements, in which case $b(i'_0)$ becomes $b(i_0)$. This is convenient for interpreting generators and differentials of $\dualSmallBarRed$ in Figure~\ref{figure:parameterization--algebra}.

		\begin{definition} \label{def:dualSmallBarRed}
		{\it Reduced small bar resolution} $ \dualSmallBarRed$ of algebra $\A$ is a type DD structure which consists of the following 24 generators (we list them with their idempotents): 

		\begin{align*}
		&_{i_2}{(b(i_2))}_{i_2},
		_{i_0}{(b(i_0))}_{i_0},
		_{j_1}{(b(j_1))}_{j_1},
		_{j_2}{(b(j_2))}_{j_2},
		_{j_0}{(b(j_0))}_{j_0},
		_{i_1}{(b(i_1))}_{i_1},\\
		&_{j_0}{(b(-\rho_0))}_{i_0},
		_{j_1}{(b(-\eta_2))}_{i_1},
		_{j_2}{(b(-\xi_2))}_{i_2},\\
		&_{j_0}{(b(-\eta_3))}_{j_1},
		_{j_1}{(b(-\xi_3))}_{j_2},
		_{j_2}{(b(-\rho_2))}_{i_2},
		_{i_2}{(b(-\xi_1))}_{i_1},
		_{i_1}{(b(-\eta_1))}_{i_0},\\
		&_{j_2}{(b(-\rho_2,-\xi_1))}_{i_1},
		_{j_1}{(b(-\xi_3,-\rho_2))}_{i_2},
		_{j_0}{(b(-\eta_3,-\xi_3))}_{j_2},
		_{i_2}{(b(-\xi_1,-\eta_1))}_{i_0},\\
		&_{j_0}{(b(-\eta_3,-\xi_3,-\rho_2))}_{i_2},
		_{j_1}{(b(-\xi_3,-\rho_2,-\xi_1))}_{i_1},
		_{j_2}{(b(-\rho_2,-\xi_1,-\eta_1))}_{i_0},\\
		&_{j_1}{(b(-\xi_3,-\rho_2,-\xi_1,-\eta_1))}_{i_0},
		_{j_0}{(b(-\eta_3,-\xi_3,-\rho_2,-\xi_1))}_{i_1},\\
		&_{j_0}{(b(-\eta_3,-\xi_3,-\rho_2,-\xi_1,-\eta_1))}_{i_0}.\\
		\end{align*}
		These are the actions of $\dualSmallBarRed$:
		$ \newline b(-\eta_3,-\xi_3,-\rho_2,-\xi_1)\longrightarrow 1 \otimes b(-\eta_3,-\xi_3,-\rho_2,-\xi_1,-\eta_1) \otimes \eta_1$, 
		$b(-\rho_2,-\xi_1,-\eta_1)\longrightarrow \xi_3 \otimes b(-\xi_3,-\rho_2,-\xi_1,-\eta_1) \otimes 1 $, 
		$b(-\eta_3)\longrightarrow 1 \otimes b(-\eta_3,-\xi_3) \otimes \xi_3 $, 
		$b(i_2)\longrightarrow \xi_2 \otimes b(-\xi_2) \otimes 1 $, 
		$b(i_2)\longrightarrow 1 \otimes b(-\xi_1) \otimes \xi_1 $, 
		$b(i_2)\longrightarrow \rho_2 \otimes b(-\rho_2) \otimes 1 $, 
		$b(j_1)\longrightarrow \eta_3 \otimes b(-\eta_3) \otimes 1 $, 
		$b(j_1)\longrightarrow 1 \otimes b(-\xi_3) \otimes \xi_3 $, 
		$b(j_1)\longrightarrow 1 \otimes b(-\eta_2) \otimes \eta_2 $, 
		$b(j_2)\longrightarrow 1 \otimes b(-\xi_2) \otimes \xi_2 $, 
		$b(j_2)\longrightarrow \xi_3 \otimes b(-\xi_3) \otimes 1 $, 
		$b(j_2)\longrightarrow 1 \otimes b(-\rho_2) \otimes \rho_2 $, 
		$b(j_0)\longrightarrow 1 \otimes b(-\rho_0) \otimes \rho_0 $, 
		$b(j_0)\longrightarrow 1 \otimes b(-\eta_3) \otimes \eta_3 $, 
		$b(-\xi_3,-\rho_2,-\xi_1,-\eta_1)\longrightarrow \eta_3 \otimes b(-\eta_3,-\xi_3,-\rho_2,-\xi_1,-\eta_1) \otimes 1 $, 
		$b(i_1)\longrightarrow \xi_1 \otimes b(-\xi_1) \otimes 1 $, 
		$b(i_1)\longrightarrow 1 \otimes b(-\eta_1) \otimes \eta_1 $, 
		$b(i_1)\longrightarrow \eta_2 \otimes b(-\eta_2) \otimes 1 $, 
		$b(-\rho_2,-\xi_1)\longrightarrow 1 \otimes b(-\rho_2,-\xi_1,-\eta_1) \otimes \eta_1 $, 
		$b(-\rho_2,-\xi_1)\longrightarrow \xi_3 \otimes b(-\xi_3,-\rho_2,-\xi_1) \otimes 1 $, 
		$b(i_0)\longrightarrow \rho_0 \otimes b(-\rho_0) \otimes 1 $, 
		$b(i_0)\longrightarrow \eta_1 \otimes b(-\eta_1) \otimes 1 $, 
		$b(-\xi_1)\longrightarrow \rho_2 \otimes b(-\rho_2,-\xi_1) \otimes 1 $, 
		$b(-\xi_1)\longrightarrow 1 \otimes b(-\xi_1,-\eta_1) \otimes \eta_1 $, 
		$b(-\eta_1)\longrightarrow \xi_1 \otimes b(-\xi_1,-\eta_1) \otimes 1 $, 
		$b(-\xi_3)\longrightarrow 1 \otimes b(-\xi_3,-\rho_2) \otimes \rho_2 $, 
		$b(-\xi_3)\longrightarrow \eta_3 \otimes b(-\eta_3,-\xi_3) \otimes 1 $, 
		$b(-\rho_2)\longrightarrow 1 \otimes b(-\rho_2,-\xi_1) \otimes \xi_1 $, 
		$b(-\rho_2)\longrightarrow \xi_3 \otimes b(-\xi_3,-\rho_2) \otimes 1 $, 
		$b(-\eta_3,-\xi_3,-\rho_2)\longrightarrow 1 \otimes b(-\eta_3,-\xi_3,-\rho_2,-\xi_1) \otimes \xi_1 $, 
		$b(-\xi_3,-\rho_2,-\xi_1)\longrightarrow \eta_3 \otimes b(-\eta_3,-\xi_3,-\rho_2,-\xi_1) \otimes 1 $, 
		$b(-\xi_3,-\rho_2,-\xi_1)\longrightarrow 1 \otimes b(-\xi_3,-\rho_2,-\xi_1,-\eta_1) \otimes \eta_1 $, 
		$b(-\xi_3,-\rho_2)\longrightarrow \eta_3 \otimes b(-\eta_3,-\xi_3,-\rho_2) \otimes 1 $, 
		$b(-\xi_3,-\rho_2)\longrightarrow 1 \otimes b(-\xi_3,-\rho_2,-\xi_1) \otimes \xi_1 $, 
		$b(-\eta_3,-\xi_3)\longrightarrow 1 \otimes b(-\eta_3,-\xi_3,-\rho_2) \otimes \rho_2 $, 
		$b(-\xi_1,-\eta_1)\longrightarrow \rho_2 \otimes b(-\rho_2,-\xi_1,-\eta_1) \otimes 1 $.
		\end{definition}

		Looking at the left of Figure~\ref{figure:parameterization--algebra}, it is convenient to see generators of $\dualSmallBarRed$ as paths (against orientation) on the boundaries of big domains $B_2$, $B_3$, $B_4$, $B_1$, encoded by the algebra elements one encounters on that path. The differential then corresponds to prolonging paths by one chord. See Figure~\ref{figure:geom_diff} for an example of this geometric interpretation of the differential. 

		\begin{figure}[!ht]
		\centering
		\includegraphics[width=0.7\textwidth]{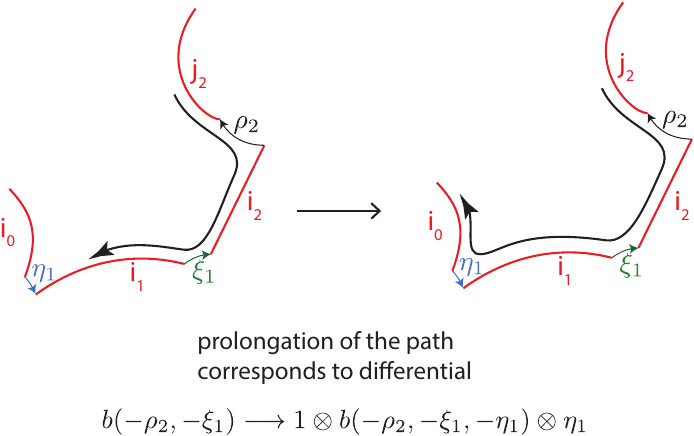}
		\caption{Geometric interpretation of the differential in $\dualSmallBarRed$.}
		\label{figure:geom_diff}
		\end{figure}

		\begin{any}{Algebraic pairing}
		First, let us refer to \cite[Section 2.3]{LOT-mor} for definitions of dual $A_\infty$ modules and type D-structures. Having an $A_\infty$ module $M_{\A}$, we denote its dual by $_{\A}\ov{M}$. Now we are ready to define algebraic pairing of curves in the pillowcase.

		\begin{definition}
		Suppose $L_0, L_1$ are two curves in the pillowcase $\P$. An {\it algebraic pairing } of curves is given by a complex
		$$M(L_1)_{\A} \boxtimes \dualSmallBarRed \boxtimes {}_{\A}\ov{M(L_0)}.$$
		\end{definition}

		\begin{remark}
		The way we constructed $\dualSmallBarRed$ ensures that $M(L_1)_{\A} \boxtimes \dualSmallBarRed \boxtimes {}_{\A}\ov{M(L_0)} \simeq Mor(M(L_0)_{\A},M(L_1)_{\A})$, see Section~\ref{sec:reasons}.
		\end{remark}
		\end{any}

\Needspace{8\baselineskip}
\section{Pairings are the same} \label{sec:thm}
	\begin{theorem}\label{pairing theorem}
	Let $L_0, L_1$ be two admissible unobstructed curves in the pillowcase $\P$. Then their Lagrangian Floer complex is homotopy equivalent to the algebraic pairing of curves: 
	$$CF_*(L_0,L_1) \simeq M(L_1)_{\A} \boxtimes \dualSmallBarRed \boxtimes {}_{\A}\ov{M(L_0)}.$$
	\end{theorem}
	\begin{proof}
	The plan for the proof is the following:
	\begin{enumerate}[label=(\arabic*)]
		\item We first isotope curves $L_0, L_1$ in a certain way.

		\item We prove that after the step (1) the pair $(L_0, L_1)$ becomes admissible.

		\item We then prove that the two chain complexes we consider are isomorphic: $CF_*(L_0,L_1)\cong M(L_1)_{\A} \boxtimes \dualSmallBarRed \boxtimes {}_{\A}\ov{M(L_0)}$.
		\begin{enumerate}[label=\textbf{\alph*})]
			\item We first see that the generators are in 1-1 correspondence.

			\item We then prove that the differentials coincide. This is done by \say{localizing} differential in the geometric pairing, i.e. by noting that every disc contributing to the differential is contained almost entirely in one of the big domains $B_1,B_2,B_3,B_4$. %
		\end{enumerate}
	\end{enumerate}
	We will be illustrating each step on our running example of curves: $L_0=L^\natural$ --- curve in Figure~\ref{figure:curve}, that corresponds to the trivial tangle $A_1 \cup A_2$, and $L_1=L^b$ --- belt around the pillowcase on the left of Figure~\ref{figure:pillowcase}. For $A_\infty$ actions on the dual module $_{\A}\ov{M(L^\natural)}$ see Figure~\ref{figure:dual_module}, and $A_\infty$ actions on ${M(L^b)_\A}=\langle x,s,z,w\rangle _{\F_2}$ (see Figure~\ref{figure:L_0+L_1}) are as follows: $z\otimes (\xi_3,\eta_3) \rightarrow x $, $z\otimes \rho_0 \rightarrow x $, $z\otimes (\eta_1,\xi_1) \rightarrow w $,  $w\otimes \xi_2 \rightarrow s $.

	\begin{any}{(1)} \label{step1}
		For a short visual description of the required isotopy one may look at Figure~\ref{figure:L_0+L_1}. Let us describe it now.

		First and foremost, we need to isotope both curves $L_0,L_1$ to make their positions nice, so that we can see modules $M(L_0)_\A,M(L_1)_\A$ geometrically. For that see Section~\ref{sec:isotopies_to_define_modules}.

		Let us describe further isotopies of $L_0$. Mark four points $b_1,b_2,b_3,b_4$ in the big domains $B_1,B_2,B_3,B_4$ like in Figure~\ref{figure:L_0+L_1}, and call them centers of the big domains. We then isotope the curve $L_0$ in the following way: first we make it intersect every red arc near its center (the centers of arcs are marked in Figure~\ref{figure:L_0+L_1}). Then we isotope $L_0$ so that in big domains it goes from the centers of big domains to the centers of red arcs straight (or, if $L_0$ is an arc, to the point on $\bdry \P$, see Figure~\ref{figure:perturbation_arcs}). See Figure~\ref{figure:L_0+L_1} for how the isotoped $L_0$ looks like.

		Concerning the curve $L_1$, we also make it intersect the red arcs near their centers. But the rest of the isotopy is different from $L_0$. First, tilt the angle in which it intersects the centers of the red arcs, so that the following is true: 1) $L_1$ is almost parallel to red arcs and intersects each nearby piece of $L_0$ exactly once; 2) going clockwise around the center of the red arc, we encounter the rays in the following order: red arc, all the pieces of $L_1$, all the pieces of $L_0$. See Figure~\ref{figure:near_center}.

		\begin{figure}[!ht]
		\centering
		\includegraphics[width=0.3\textwidth]{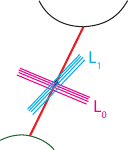}
		\caption{Perturbation near the centers of red arcs.}
		\label{figure:near_center}
		\end{figure}

		Next, we make the final isotopy of the $L_1$ curve, which has to do with the way it behaves inside the big domains. Divide $L_1$ on the segments by intersections with red arcs. We already specified how $L_1$ looks near those intersections. Now we will describe how each segment between those intersections is isotoped, by traversing $L_1$. First, an important thing to note, the whole $L_1$ will not leave the small neighborhood of $\bdry \P \cup \{\text{red arcs}\}$. We start at the center of the red arc, enter one of the big domains, and then go near the $\bdry \P \cup \{\text{red arcs}\}$ in that domain until it reaches a basepoint. If this is the end segment of $L_1$ being an arc, then $L_1$ is connected to $\bdry P$ near that basepoint, such that $(L_0,L_1)$ is admissible, see Figure~\ref{figure:perturbation_arcs}. Otherwise $L_1$ turns by $360^\circ$ (in the direction towards the other end of the segment, i.e. such that it does not introduce a fishtail), and goes backwards until it reaches the other end of the segment.  See Figure~\ref{figure:L_0+L_1} for how the isotoped $L_1$ looks like.

		\begin{figure}[ht]
		\centering
		\includegraphics[width=0.45\textwidth]{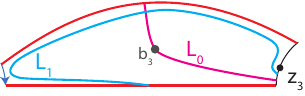}
		\caption{Perturbation of the end segments of the arcs.}
		\label{figure:perturbation_arcs}
		\end{figure}
	\end{any}

	\begin{figure}[H]
	{
	\setlength{\fboxsep}{15pt}%
	\fbox{\includegraphics[width=0.5\textwidth]{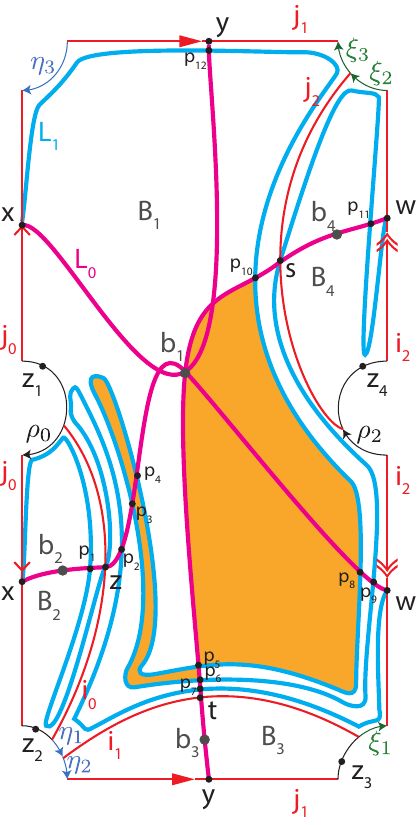}}
	}
	\caption{Isotopies of $L_0=L^\natural$ and $L_1=L^b$ so that the chain complexes of geometric and algebraic pairings become isomorphic: $CF_*(L_0,L_1)\cong M(L_1)_{\A} \boxtimes \dualSmallBarRed \boxtimes {}_{\A}\ov{M(L_0)}$.}
	\label{figure:L_0+L_1}
	\end{figure}

	\begin{any}{(2)}
		Let us prove that after step (1) the pair $(L_0,L_1)$ becomes admissible. All the non-smooth corners, triple intersections, non-transverse intersections are eliminated by introducing a slight perturbation. There are no immersed annuli because of intersections introduced in Figure~\ref{figure:near_center}. If $p L_0 \sim q L_1 $, and so $Per(L_0,L_1)=\Z$, those intersections lift to the covering from the Lemma~\ref{lemma:non-admiss_iff_positive_shadow_iff_no_intersection},d).
		If $L_0,L_1$ are arcs, they are in admissible position relative to basepoints because we ensured it while isotoping $L_1$, see Figure~\ref{figure:perturbation_arcs}.
	\end{any}

	\begin{any}{(3a)}
		The generators of $CF_*(L_0,L_1)$, as well as the generators of $M(L_1)_{\A} \boxtimes \dualSmallBarRed \boxtimes {}_{\A}\ov{M(L_0)}$, are in 1-1 correspondence with the set of paths along $\bdry B_i$, from intersections $L_1 \cap \{i_0, i_1, i_2, j_0, j_1, j_2\}=\{\text{generators of }M(L_1)\}$ to intersections $L_0 \cap \{i_0, i_1, i_2, j_0, j_1, j_2\}=\{\text{generators of }\ov{M(L_0)}\}$. The paths are against natural orientations of $\bdry B_i$, and consist of chords $-\gamma$ of length one (which are also elements of $\A$). Let us explain how to see those paths.

		Remember that elements of $\dualSmallBarRed$ naturally correspond to such paths, see Figure~\ref{figure:geom_diff}. For the generators in $M(L_1)_{\A} \boxtimes \dualSmallBarRed \boxtimes {}_{\A}\ov{M(L_0)}$, they each have their element of $\dualSmallBarRed$ in the center, and that describes the path from a generator in  $M(L_1)$ to generator in $M(L_0)$. For the generators of $CF_*(L_0,L_1)$, the desired path can be traversed along the $L_1$, see Figure~\ref{figure:L_0+L_1}. Notice that we include \say{0 length} paths, which correspond to intersections introduced when both $L_0,L_1$ cross the same red arc, see Figure~\ref{figure:near_center}.

		Considering example in Figure~\ref{figure:L_0+L_1}, we have:
		$$M(L_1)=\langle x,s,z,w\rangle_{\F_2}, \qquad \ov{M(L_0)}=\langle x^*,s^*,z^*,w^*,t^*,y^*\rangle_{\F_2}.$$
		The intersection points $L_0 \cap L_1$ in Figure~\ref{figure:L_0+L_1} correspond to the generators of $M(L_1)_{\A} \boxtimes \dualSmallBarRed \boxtimes {}_{\A}\ov{M(L_0)}$ in the following way:
	
		\begin{align*}
			&w\boxtimes b(i_2)\boxtimes w^* \leftrightarrow w, \\
			&s\boxtimes b(j_2)\boxtimes s^* \leftrightarrow s, \\
			&x\boxtimes b(j_0)\boxtimes x^* \leftrightarrow x, \\
			&z\boxtimes b(i_0)\boxtimes z^* \leftrightarrow z, \\
			&s\boxtimes b(-\rho_2,-\xi_1,-\eta_1)\boxtimes z^*  \leftrightarrow p_3, \\
			&s\boxtimes b(-\rho_2)\boxtimes w^* \leftrightarrow p_9, \\
			&s\boxtimes b(-\xi_2)\boxtimes w^* \leftrightarrow p_{11}, \\
			&s\boxtimes b(-\rho_2,-\xi_1)\boxtimes t^* \leftrightarrow p_6, \\
			&w\boxtimes b(-\xi_1,-\eta_1)\boxtimes z^*  \leftrightarrow p_2, \\
			&w\boxtimes b(-\xi_1)\boxtimes t^* \leftrightarrow p_7, \\
			&x\boxtimes b(-\eta_3)\boxtimes y^* \leftrightarrow p_{12}, \\
			&x\boxtimes b(-\eta_3,-\xi_3,-\rho_2,-\xi_1,-\eta_1)\boxtimes z^* \leftrightarrow p_4, \\
			&x\boxtimes b(-\eta_3,-\xi_3)\boxtimes s^* \leftrightarrow p_{10}, \\
			&x\boxtimes b(-\eta_3,-\xi_3,-\rho_2,-\xi_1)\boxtimes t^* \leftrightarrow p_5, \\
			&x\boxtimes b(-\eta_3,-\xi_3,-\rho_2)\boxtimes w^* \leftrightarrow p_8, \\
			&x\boxtimes b(-\rho_0)\boxtimes z^*  \leftrightarrow p_1. \\
		\end{align*}
	\end{any}

	\begin{any}{(3b)}
		Here we want to show that differentials in $CF_*(L_0,L_1)$ and $M(L_1)_{\A} \boxtimes \dualSmallBarRed \boxtimes {}_{\A}\ov{M(L_0)}$ coincide. We will do it by partitioning both differentials into smaller groups, and showing how smaller groups correspond to each other.

		\begin{lemma}\label{lem:first}
		Every immersed disc contributing to differential in $CF_*(L_0,L_1)$ is contained inside a small neighborhood of one of the big domains $B_1,B_2,B_3,B_4$.
		\end{lemma}
		\begin{proof}
		For an immersed disc to go from one big domain to another big domain, it must pass through an intersection of type $q \boxtimes b(i)\boxtimes k^*$, because the disc is not allowed to touch the $\bdry \P$. Here, in the notation, we use 1-1 correspondence between the generators of $CF_*(L_0,L_1)$ and the generators of $M(L_1)_{\A} \boxtimes \dualSmallBarRed \boxtimes {}_{\A}\ov{M(L_0)}$. Such intersections happen when both $L_0$ and $L_1$ cross the red arc, i.e. they correspond to \say{0 length} paths, see Figure~\ref{figure:near_center}. Also only two opposite parts of the corner $q \boxtimes b(i)\boxtimes k^* \in L_0 \cap L_1$ are allowed to be filled by the disc. Thus the disc cannot pass through such intersection point, as such disc cannot be immersed. %
		\end{proof}
		\begin{remark}
		We need to consider small neighborhoods of the big domains, as intersections of type $q \boxtimes b(i)\boxtimes k^*$ are not happening exactly on the red arc, but rather somewhere close to its center, see Figure~\ref{figure:near_center}.
		\end{remark}

		\begin{lemma}\label{lem:second}
		Every differential in $M(L_1)_{\A} \boxtimes \dualSmallBarRed \boxtimes {}_{\A}\ov{M(L_0)}$ contains an $A_\infty$ action either on the $M(L_1)$ side (Figure~\ref{figure:diff_1type}), or on the $\ov{M(L_0)}$ side (Figure~\ref{figure:diff_2type}), but not on both sides. Moreover, every such $A_\infty$ action comes from a \say{basic} disc in the definition of $M(L_i)_{\A}$, i.e. a disc contained entirely in one of the big domains.
		\end{lemma}
		\begin{proof}
		The first observation is that $\dualSmallBarRed$ does not have differentials with algebra elements outgoing on both sides:
		$$
		\begin{tikzpicture}[scale=0.4]
			\node[](in1) at (1.4,2){};
			\node[circle,fill,scale=0.3](diff) at (1.4,1.2){};
			\node[](out0) at (0,0){};
			\node[](out1) at (1.4,0){};
			\node[](out2) at (3,0){};
			\draw[arrow] (in1) to (out1);
			\draw[arrow] (diff) to[bend left] (out2);
			\draw[arrow] (diff) to[bend right] (out0);
		\end{tikzpicture}
		$$
		This implies that the chain complex structure does not depend on the brackets placement:
		$$(M(L_1)_{\A} \boxtimes \dualSmallBarRed) \boxtimes {}_{\A}\ov{M(L_0)} =M(L_1)_{\A} \boxtimes (\dualSmallBarRed \boxtimes {}_{\A}\ov{M(L_0)}) =M(L_1)_{\A} \boxtimes \dualSmallBarRed \boxtimes {}_{\A}\ov{M(L_0)}.$$ 
		Usually only the homotopy type of the box tensor product does not depend on the brackets placement. Also, it implies that the differential is either on the $M(L_1)$ side (Figure~\ref{figure:diff_1type}), or on the $\ov{M(L_0)}$ side (Figure~\ref{figure:diff_2type}). We call them the 1st and the 2nd types of differentials.

		The second observation is that the differentials in $\dualSmallBarRed$ do not contain outgoing algebra elements of chord length more than one (an example of chord length two algebra element is $\xi_{12}$). This observation implies the last statement of the lemma.
		\end{proof}

		Now let us take one connected segment $l_0$ of $L_0$, which is cut out by a small neighborhood of a big domain  $N(B_k)$. And also we take one connected segment $l_1$ of $L_1$, which is cut out by the same neighborhood $N(B_k)$. These segments almost coincide with two of the segments from the division of $L_0$ and $L_1$ by the intersections with red arcs. Their behavior inside $N(B_k)$ is completely described by our isotopy in step (1). Also note, that such segments correspond to basic discs in the definition of $M(L_i)_\A$. These basic discs 
		have a chance to contribute to the differential in $M(L_1)_{\A} \boxtimes \dualSmallBarRed \boxtimes {}_{\A}\ov{M(L_0)}$.

		Due to Lemma~\ref{lem:first}, the differential in $CF_*(L_0,L_1)$ is partitioned into differentials with boundaries on segments $l_0,l_1$. We will denote such groups of differentials by $CF_*(l_0,l_1)$. Due to Lemma~\ref{lem:second}, the differential in $M(L_1)_{\A} \boxtimes \dualSmallBarRed \boxtimes {}_{\A}\ov{M(L_0)}$ is partitioned into differentials using different basic discs. We will denote such groups of differentials by $M(l_1)_{\A} \boxtimes \dualSmallBarRed \boxtimes {}_{\A}\ov{M(l_0)}$, as basic discs correspond to segments.

		We are left to show how differentials in $CF_*(l_0,l_1)$ correspond to differentials in $M(l_1)_{\A} \boxtimes \dualSmallBarRed \boxtimes {}_{\A}\ov{M(l_0)}$. We will do it by considering the 1st and the 2nd type of differentials in $M(l_1)_{\A} \boxtimes \dualSmallBarRed \boxtimes {}_{\A}\ov{M(l_0)}$ separately. 
		\newline \newline
		\underline{The 1st type.} In this case the differential $M(l_1)_{\A} \boxtimes \dualSmallBarRed \boxtimes {}_{\A}\ov{M(l_0)}$ has outgoing algebra elements on the left, see Figure~\ref{figure:diff_1type}. This corresponds to prolongation of the path in the backward direction, i.e. new length one chords are concatenated to the path from the left. We use the following notation in Figure~\ref{figure:diff_1type}: $i,k$ are the red arcs intersecting $l_1$ at $(u_1)$ and $(u_2)$, $j$ is the red arc intersecting $l_0$ at $(v)$, and $\gamma_m$ represent chord length one elements of $\A$.

		\begin{figure}[!ht]
		\centering
		\begin{tikzpicture}[scale=1.4]
			\node[](in0) at (-1.5,4){${(u_1)}_i$};
			\node[](in1) at (1,4){$_i{b(path)}_j$};
			\node[](in2) at (3.5,4){$_j{(v)}$};

			\node[]() at (-0.25,4){$\boxtimes$};
			\node[]() at (2.25,4){$\boxtimes$};
			\node[]() at (-1,0){$\boxtimes$};
			\node[]() at (3,0){$\boxtimes$};
			
			\node[](out0) at (-1.5,0){${(u_2)}_k$ };
			\node[](out1) at (1,0){$_k{b(-\gamma_m , \ldots , -\gamma_2,  -\gamma_1, path)}_j$};
			\node[](out2) at (3.5,0){$_j{(v)}$};
			
			\draw[arrow] (in0) to (out0);
			\draw[arrow] (in1) to (out1);
			\draw[arrow] (in2) to (out2);

			\node[circle,fill,scale=0.3](ldiff) at (-1.5,0.9){};
			\node[circle,fill,scale=0.3](rdiff1) at (1,3.3){};
			\node[circle,fill,scale=0.3](rdiff2) at (1,2.4){};
			\node[circle,fill,scale=0.3](rdiff3) at (1,1.6){};
			\node[](dots) at (0.8,1.9){$\vdots$};
			\draw[] (rdiff1) to node[above,sloped, pos=0.35]{$\gamma_1$}  (ldiff);
			\draw[] (rdiff2) to node[above,sloped, pos=0.35]{$\gamma_2$} (ldiff);
			\draw[] (rdiff3) to node[below,sloped, pos=0.4]{$\gamma_m$} (ldiff);
		\end{tikzpicture}
		\caption{The 1st type of differentials in $M(L_1)_{\A} \boxtimes \dualSmallBarRed \boxtimes {}_{\A}\ov{M(L_0)}$.}
		\label{figure:diff_1type}
		\end{figure}

		Let us describe the corresponding disc differentials in $CF(l_0,l_1)$. Suppose the disc goes from $p$ to $q$. First, note that all intersections $l_0 \cap l_1$ are happening near one of the two ends of segment $l_0$. Points $p$ and $q$ can be on one end of the segment $l_0$, or on the different ends. Let us consider those pairs of point $p$ and $q$, which are on one end of the segment $l_0$. For this to happen, traversing the $l_1$ boundary of the disc, $l_1$ must pass the $360^\circ$ rotation point and come back. See, for example, the disc from $p_6$ to $p_5$ in Figure~\ref{figure:L_0+L_1}. We say that such differentials are of the 1st type in $CF(l_0,l_1)$. 

		These are precisely the discs that correspond to the 1st type of differentials in $M(l_1)_{\A} \boxtimes \dualSmallBarRed \boxtimes {}_{\A}\ov{M(l_0)}$. The reason is that both the 1st type of differentials in $M(l_1)_{\A} \boxtimes \dualSmallBarRed \boxtimes {}_{\A}\ov{M(l_0)}$, and the 1st type of differentials in $CF(l_0,l_1)$ exist if and only if $(u_2)_k$ is before $(u_1)_i$ is before $_j(v)$ with respect to the basepoint and the direction against the natural orientation of $\bdry B_i$.
		For example, in Figure~\ref{figure:L_0+L_1}, the disc from $p_6$ to $p_5$ corresponds to differential $s\boxtimes b(-\rho_2,-\xi_1)\boxtimes t^* \rightarrow x\boxtimes b(-\eta_3,-\xi_3,-\rho_2,-\xi_1)\boxtimes t^*$.
		\newline \newline
		\underline{The 2nd type.}
		In this case the differential $M(l_1)_{\A} \boxtimes \dualSmallBarRed \boxtimes {}_{\A}\ov{M(l_0)}$ has outgoing algebra elements on the right, see Figure~\ref{figure:diff_2type}. This corresponds to prolongation of the path in the forward direction, i.e. new length one chords are concatenated to the path from the right. We use the following notation in Figure~\ref{figure:diff_2type}: $i$ is the red arc intersecting $l_1$ at $(u)$, $j$, $o$ are the red arcs intersecting $l_0$ at $(v_1)$ and $(v_2)$, and $\gamma_m$ represent chord length one elements of $\A$.

		\begin{figure}[!ht]
		\centering
		\begin{tikzpicture}[scale=1.4]
			\node[](in0) at (-1.5,4){${(u)}_i$};	
			\node[](in1) at (1,4){$_i{b(path)}_j$};
			\node[](in2) at (3.5,4){$_j{(v_1)}$};

			\node[]() at (-0.25,4){$\boxtimes$};
			\node[]() at (2.25,4){$\boxtimes$};
			\node[]() at (-1,0){$\boxtimes$};
			\node[]() at (3,0){$\boxtimes$};
			
			\node[](out0) at (-1.5,0){${(u)}_i$ };
			\node[](out1) at (1,0){$_i{b( path,-\gamma_1, -\gamma_2, \ldots,-\gamma_m)}_o$};
			\node[](out2) at (3.5,0){$_o{(v_2)}$};
			
			\draw[arrow] (in0) to (out0);
			\draw[arrow] (in1) to (out1);
			\draw[arrow] (in2) to (out2);

			\node[circle,fill,scale=0.3](ldiff) at (3.5,0.9){};
			\node[circle,fill,scale=0.3](rdiff1) at (1,3.3){};
			\node[circle,fill,scale=0.3](rdiff2) at (1,2.4){};
			\node[circle,fill,scale=0.3](rdiff3) at (1,1.6){};
			\node[](dots) at (1.2,1.9){$\vdots$};
			\draw[] (rdiff1) to node[above,sloped, pos=0.35]{$\gamma_1$}  (ldiff);
			\draw[] (rdiff2) to node[above,sloped, pos=0.35]{$\gamma_2$} (ldiff);
			\draw[] (rdiff3) to node[below,sloped, pos=0.4]{$\gamma_m$} (ldiff);
		\end{tikzpicture}
		\caption{The 2nd type of differentials in $M(L_1)_{\A} \boxtimes \dualSmallBarRed \boxtimes {}_{\A}\ov{M(L_0)}$.}
		\label{figure:diff_2type}
		\end{figure}
		The corresponding 2nd type of disc differentials in $CF(l_0,l_1)$ are those, which have their corners on two different ends of segment $l_0$. They do not pass through the $360^\circ$ rotation point of $l_1$, but instead they pass through the center of the big domain, see the disc from $p_{10}$ to $p_5$ in Figure~\ref{figure:L_0+L_1}. 

		Both the 2nd type of differentials in $M(l_1)_{\A} \boxtimes \dualSmallBarRed \boxtimes {}_{\A}\ov{M(l_0)}$, and the 2nd type of differentials in $CF(l_0,l_1)$ exist if and only if $(u)_i$ is before $_j{(v_1)}$ is before $_o{(v_2)}$ with respect to the basepoint and the direction against the natural orientation of $\bdry B_i$. For example, in Figure~\ref{figure:L_0+L_1}, the disc from $p_{10}$ to $p_5$ corresponds to the differential $ x\boxtimes b(-\eta_3,-\xi_3)\boxtimes s^* \rightarrow x\boxtimes b(-\eta_3,-\xi_3,-\rho_2,-\xi_1)\boxtimes t^*$.
	\end{any}
	\end{proof}

\Needspace{8\baselineskip}
\section{Modules associated to tangles}
	Here we describe examples of modules $M(L_{(K,\S,\pi)})_{\A}$ associated to tangles $K \setminus (A_1 \cup A_2)$. We use calculations from \cite[Sections 7,11]{HHK2} to get immersed curves $L_{(K,\S,\pi)}$ in the pillowcase, see Figure~\ref{figure:curves}.  Also, for the sake of readability of this section we will abuse notation and write $L_{K}$ instead of $L_{(K,\S,\pi)}$.

	\begin{example}[Trivial tangle to pair with]\label{example:triv_tangle}
	First we consider a trivial tangle $A_1 \cup A_2$ inside the Conway sphere, see the left picture of the second row in Figure~\ref{figure:strategy}. For that tangle (decorated with an additional arc and a circle to avoid reducibles) we associate a curve $L^{\natural}$ in Figure~\ref{figure:curve}, and to that curve we associate a module described in Figure~\ref{figure:module}. 

	Because in the algebraic pairing $M(L_K)_{\A} \boxtimes \dualSmallBarRed \boxtimes {}_{\A}\ov{M(L^{\natural})}$ there is a dual module $_{\A}\ov{M(L^{\natural})}$ involved, we describe it here:

	\begin{figure}[!ht]
	\centering %
	\begin{tikzpicture}[scale=0.4,baseline=1.5cm]
		\foreach \a/\b in {1/w*, 2/z*, 3/x*, 4/y*, 5/t*, 6/s*}{
		\draw (\a*360/6: 4cm) node(\b){\b};
		}

		\draw[arrow] (x*) to node[left]{$\rho_0$} (z*);
		\draw[arrow] (x*) to node[left]{$\eta_3$} (y*);
		\draw[arrow] (y*) to node[below]{$\eta_2$} (t*);
		\draw[arrow] (s*) to node[right]{$(\xi_1,\rho_2)$} (t*);
		\draw[arrow] (s*) to node[right]{$\xi_2$} (w*);
		\draw[arrow] (w*) to node[above]{$(\eta_1,\xi_1)$} (z*);

		\draw[arrow] (s*) to node[below]{$(\eta_1,\xi_{12})$} (z*);
		\draw[arrow] (x*) to node[above]{$\eta_{23}$} (t*);
		
	\end{tikzpicture}
	\caption{$_{\A}\ov{M(L^{\natural})}$}
	\label{figure:dual_module}
	\end{figure}

	\end{example} 

	In the next examples we also compute pillowcase homology via the algebraic pairing, using computer program \cite{Pyt} for box tensor product of modules. Another way to see the chain complex and the differentials is to isotope $L_{K}$ as in the proof of Theorem~\ref{pairing theorem}, and then to use the Lagrangian Floer chain complex $CF_*(L^{\natural},L_{K})$.

	\begin{example}[The unknot]\label{example:unknot}
	The next example is a trivial knot tangle $U-A_1-A_2$. Depending on how we pick the second tangle for the trivial knot (the first tangle is $A_1 \cup A_2$), the resulting curve on the pillowcase can be different. In the two simplest cases it can be an arc $\{\gamma=\pi\}$ (in case the second tangle looks like a crossing $\times$), or an arc $\{\theta=0\}$ (in case the  second tangle is horizontal smoothing of that crossing). Note that we cannot pick a vertical smoothing $)($ of a crossing, as it results in two circles if paired with $A_1 \cup A_2$.  On the left of Figure~\ref{figure:curves} we depicted an arc $\{\theta=0\}=L_U$ for a crossing $\times$. 

	The corresponding module $M(L_U)_{\A}$ is $q_{j_1}\xrightarrow{\eta_3}p_{j_0}$. The algebraic pairing chain complex $M(L_U)_{\A} \boxtimes \dualSmallBarRed \boxtimes {}_{\A}\ov{M(L^{\natural})}$ has 13 generators and 12 differentials. Pillowcase homology then has rank one: $HF_*(L^{\natural},L_U)=H_*(M(L_U)_{\A} \boxtimes \dualSmallBarRed \boxtimes {}_{\A}\ov{M(L^{\natural})})=\F_2$, which coincides with singular instanton knot homology $I^{\natural}(U)$.

	\end{example}
	\begin{example}[$T(2,3)$]
	An immersed curve for the right-handed trefoil is depicted on the left of Figure~\ref{figure:curves}. 

	The corresponding module $M(L_{T(2,3)})_{\A}$ has generators:
	$$u_{i_0},e_{j_1},v_{j_1},q_{i_1},$$
	and actions:
	$$u \otimes (\eta_1,\xi_1,\rho_2,\xi_3) \longrightarrow e, \ q \otimes (\eta_2) \longrightarrow e, \ q \otimes (\xi_1,\rho_2,\xi_3) \longrightarrow v.$$
	The algebraic pairing chain complex $M(L_U)_{\A} \boxtimes \dualSmallBarRed \boxtimes {}_{\A}\ov{M(L^{\natural})}$ has 15 generators and 10 differentials. Pillowcase homology then has rank three: $HF_*(L^{\natural},L_{T(2,3)})=H_*(M(L_{T(2,3)})_{\A} \boxtimes \dualSmallBarRed \boxtimes {}_{\A}\ov{M(L^{\natural})})=(\F_2)^3$, which coincides with singular instanton knot homology $I^{\natural}(T(2,3))$.
	
	\end{example}

	In the next three examples immersed curves are unions of curves $R_0,R_1,R_3,R_4$, see the right of Figure~\ref{figure:curves}. Notice that $R_3$ differs from $R_0$ by a twist around the boundary, and thus their pairings with $L^{\natural}$ are the same (because  $L^{\natural}$ is not an arc).  We describe the modules $M(L_{R_{i}})_\A$ for $i=0,1,4$ in the appendix. Using those modules we compute three more examples for tangles:

	\begin{example}[$T(3,7)$]
	The corresponding immersed curve is depicted on the right of Figure~\ref{figure:curves}.

	The corresponding module is 
	$M(L_{T(3,7)})= M(R_0) \oplus M(R_1) \oplus M(R_1)$. Pillowcase homology then has rank 9: $HF_*(L^{\natural},L_{T(3,7)})=H_*(M(L_{T(3,7)})_{\A} \boxtimes \dualSmallBarRed \boxtimes {}_{\A}\ov{M(L^{\natural})})=(\F_2) \oplus (\F_2)^4 \oplus (\F_2)^4 =(\F_2)^9$, which coincides with singular instanton knot homology $I^{\natural}(T(3,7))$.
	\end{example}

	\begin{example}[$T(5,11)$]
	The corresponding immersed curve is depicted on the right of Figure~\ref{figure:curves}.

	The corresponding module is $M(L_{T(5,11)})  =  M(R_0) \oplus M(R_1) \oplus M(R_1) \oplus M(R_4) \oplus M(R_4) $. Pillowcase homology then has rank 17: $HF_*(L^{\natural},L_{T(5,11)})=H_*(M(L_{T(5,11)})_{\A} \boxtimes \dualSmallBarRed \boxtimes {}_{\A}\ov{M(L^{\natural})})=(\F_2) \oplus (\F_2)^4 \oplus (\F_2)^4 \oplus (\F_2)^4 \oplus (\F_2)^4 =(\F_2)^{17}$. 

	Singular instanton Floer homology is not known for $T(5,11)$.
	\end{example}

	\begin{example}[$T(3,4)$]
	The corresponding immersed curve is depicted on the right of Figure~\ref{figure:curves}. This is an example where we actually need to perturb $L_{T(3,4)}$ in order to get an immersed 1-manifold.

	The corresponding module is $M(L_{T(3,4)})= M(R_1) \oplus M(L_{R_3})$. Pillowcase homology then has rank 5: $HF_*(L^{\natural},L_{T(3,4)})=H_*(M(L_{T(3,4)})_{\A} \boxtimes \dualSmallBarRed \boxtimes {}_{\A}\ov{M(L^{\natural})})=(\F_2)^4 \oplus \F_2=(\F_2)^5$, which coincides with singular instanton knot homology $I^{\natural}(T(3,4))$.
	\end{example}

	See \cite{HHK2} and \cite{FKP} for other examples of immersed curves associated to tangles.
	
	\begin{figure}[h]
	\includegraphics[width=1.2\textwidth]{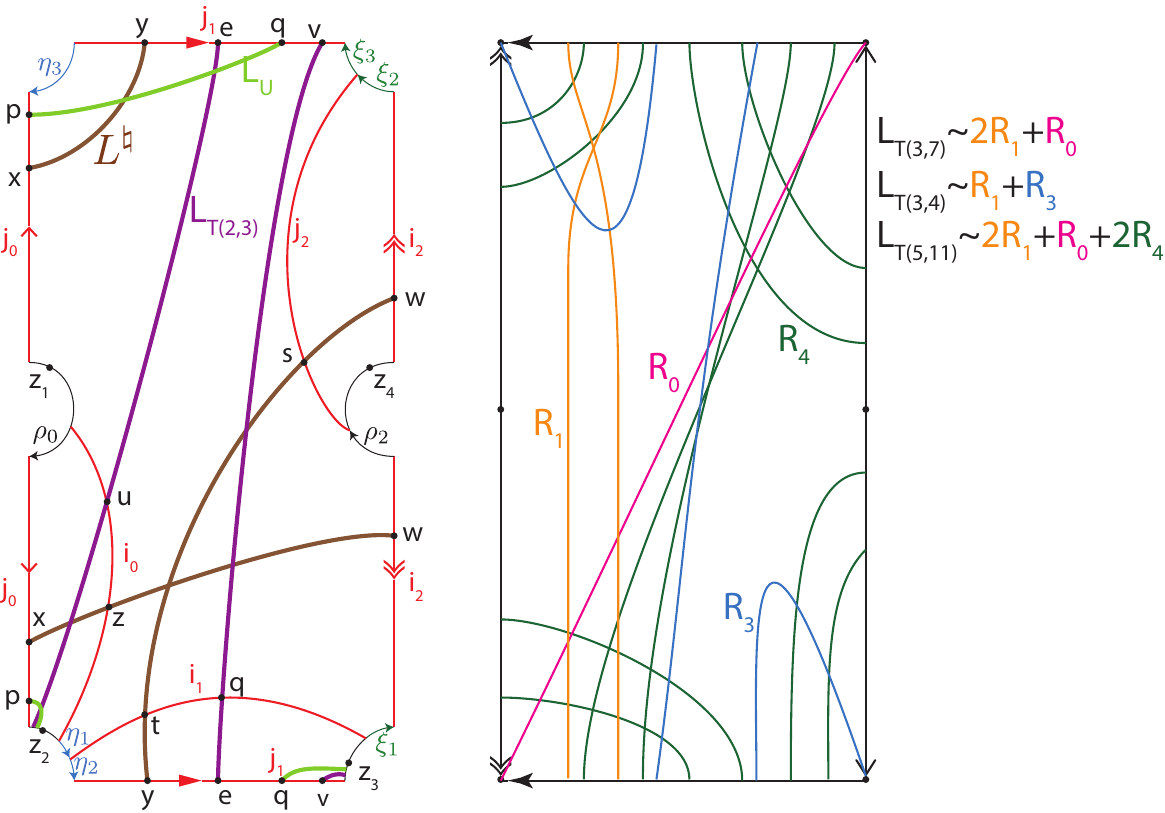}
	\caption{Different immersions associated to tangles. $L^{\natural}$ denotes an immersed curve associated to trivial tangle consisting of two arcs $A_1,A_2$, see Figure~\ref{figure:strategy}. $L_{K}$ denotes an immersed curve associated to tangle $K\setminus(A_1 \cup A_2).$}
	\label{figure:curves}
	\end{figure}

\Needspace{8\baselineskip}
\section{Appendix}
	{
	\setlength{\parindent}{0pt}
		\begin{any}{Module}

		$M(R_0)_\A$ 

		\vspace{0.2cm}
		4 generators with their idempotents:

		$a_{j_1}$,$c_{j_1}$,$b_{i_1}$,$d_{j_0}$.

		\vspace{0.2cm}

		Actions:

		$a\otimes \eta_3\longrightarrow d$,
		$b\otimes \eta_{23}\longrightarrow d$,
		$b\otimes (\xi_1, \rho_2, \xi_3) \longrightarrow  c$,
		$b\otimes \eta_2 \longrightarrow  a$
		.

		Algebraic pairing with the trivial tangle module:
		$$H_*(M(R_0)_{\A} \boxtimes \dualSmallBarRed \boxtimes {}_{\A}\ov{M(L^{\natural})})=\F_2.$$
		\end{any}

		\begin{any}{Module}

		$M(R_1)_\A$ 

		\vspace{0.2cm}
		4 generators with their idempotents:

		$x1_{j_1}$,$y1_{j_1}$,$z1_{i_1}$,$t1_{i_1}$.

		\vspace{0.2cm}

		Actions:
		$z1\otimes (\xi_1, \rho_2, \xi_3)\longrightarrow  y1$,
		$t1\otimes \eta_2\longrightarrow  y1$,
		$z1\otimes \eta_2\longrightarrow  x1$,
		$t1\otimes (\xi_1, \rho_2, \xi_3)\longrightarrow  x1$
		.

		Algebraic pairing with the trivial tangle module:
		$$H_*(M(R_1)_{\A} \boxtimes \dualSmallBarRed \boxtimes {}_{\A}\ov{M(L^{\natural})})=(\F_2)^4.$$
		\end{any}

		\begin{any}{Module}

		$M(R_4)_\A$ 

		\vspace{0.2cm}
		4 generators with their idempotents:

		$a_{i_0}$,
		$c_{i_1}$,
		$b_{i_0}$,
		$e_{i_1}$,
		$d_{i_1}$,
		$g_{i_1}$,
		$h_{i_1}$,
		$m_{j_1}$,
		$l_{j_1}$,
		$q_{j_2}$,
		$p_{j_2}$,
		$s_{i_2}$,
		$r_{i_1}$,
		$u_{j_1}$,
		$t_{i_2}$,
		$w_{j_0}$,
		$v_{j_1}$,
		$y_{j_1}$,
		$x_{j_1}$,
		$z_{j_0}$.

		\vspace{0.2cm}

		Actions:

		$a\otimes \eta_3\longrightarrow d$,
		$b\otimes \eta_{23}\longrightarrow d$,
		$b\otimes (\xi_1, \rho_2, \xi_3) \longrightarrow  c$,
		$b\otimes \eta_2 \longrightarrow  a$
		
		$p \otimes \xi_3 \longrightarrow u$,
		$t \otimes \xi_2 \longrightarrow q$,
		$d \otimes (\xi_1, \rho_2, \xi_3) \longrightarrow l$,
		$y \otimes \eta_3 \longrightarrow z$,
		$a \otimes \eta_1 \longrightarrow g$,
		$s \otimes \xi_2 \longrightarrow p$,
		$c \otimes \eta_2 \longrightarrow x$,
		$r \otimes \xi_{123} \longrightarrow u$,
		$e \otimes \eta_2 \longrightarrow m$,
		$d \otimes \eta_2 \longrightarrow y$,
		$s \otimes \xi_{23} \longrightarrow u$,
		$h \otimes \eta_2 \longrightarrow v$,
		$c \otimes \eta_{23} \longrightarrow w$,
		$r \otimes \xi_12 \longrightarrow p$,
		$q \otimes \xi_3 \longrightarrow v$,
		$r \otimes \xi_1 \longrightarrow s$,
		$a \otimes \rho_0 \longrightarrow w$,
		$t \otimes \xi_{23} \longrightarrow v$,
		$b \otimes \eta_1 \longrightarrow h$,
		$b \otimes \eta_12 \longrightarrow v$,
		$d \otimes \eta_{23} \longrightarrow z$,
		$r \otimes \eta_2 \longrightarrow l$,
		$x \otimes \eta_3 \longrightarrow w$,
		$c \otimes (\xi_1, \rho_2, \xi_3) \longrightarrow m$,
		$e \otimes \xi_12 \longrightarrow q$,
		$a \otimes \eta_12 \longrightarrow u$,
		$b \otimes \rho_0 \longrightarrow z$,
		$e \otimes \xi_123 \longrightarrow v$,
		$g \otimes \eta_2 \longrightarrow u$,
		$e \otimes \xi_1 \longrightarrow t$
		.

		Algebraic pairing with the trivial tangle module:
		$$H_*(M(R_4)_{\A} \boxtimes \dualSmallBarRed \boxtimes {}_{\A}\ov{M(L^{\natural})})=(\F_2)^4.$$
		\end{any}

		\begin{any}{Bimodule}\label{description:dual_small_bar}

		$\dualSmallBar$

		\vspace{0.2cm}
		56 generators with their idempotents:

		$_{j_1}{(b(\xi'_{32},\xi'_1))}_{i_1}$,
		$_{j_0}{(b(\rho'_0))}_{i_0}$,
		$_{j_1}{(b(\xi'_3,\xi'_2))}_{i_2}$,
		$_{j_0}{(b(\eta'_3,\xi'_3,\rho'_2,\xi'_1))}_{i_1}$,
		$_{j_0}{(b(\eta'_3,\xi'_{32}))}_{i_2}$,
		$_{j_0}{(b(\eta'_3,\xi'_{32},\xi'_1,\eta'_1))}_{i_0}$,
		$_{j_0}{(b(\eta'_3,\xi'_3,\xi'_2))}_{i_2}$,
		$_{j_1}{(b(\eta'_{21}))}_{i_0}$,
		$_{j_0}{(b(\eta'_3,\xi'_3,\xi'_2,\xi'_1,\eta'_1))}_{i_0}$,
		$_{j_2}{(b(\xi'_2))}_{i_2}$,
		$_{j_2}{(b(\xi'_2,\xi'_1,\eta'_1))}_{i_0}$,
		$_{j_1}{(b(\xi'_{32}))}_{i_2}$,
		$_{j_2}{(b(\xi'_{21},\eta'_1))}_{i_0}$,
		$_{j_0}{(b(\eta'_{321}))}_{i_0}$,
		$_{j_2}{(b(\rho'_2,\xi'_1,\eta'_1))}_{i_0}$,
		$_{j_0}{(b(\eta'_3))}_{j_1}$,
		$_{j_1}{(b(\xi'_3,\xi'_2,\xi'_1,\eta'_1))}_{i_0}$,
		$_{j_1}{(b(\xi'_{321},\eta'_1))}_{i_0}$,
		$_{i_2}{(b(i_2'))}_{i_2}$,
		$_{j_0}{(b(\eta'_3,\xi'_3,\rho'_2,\xi'_1,\eta'_1))}_{i_0}$,
		$_{j_1}{(b(\xi'_3,\xi'_{21}))}_{i_1}$,
		$_{j_1}{(b(j_1'))}_{j_1}$,
		$_{j_0}{(b(\eta'_3,\xi'_3,\xi'_{21}))}_{i_1}$,
		$_{j_2}{(b(j_2'))}_{j_2}$,
		$_{j_0}{(b(j_0'))}_{j_0}$,
		$_{j_0}{(b(\eta'_3,\eta'_2))}_{i_1}$,
		$_{j_1}{(b(\xi'_3,\rho'_2,\xi'_1,\eta'_1))}_{i_0}$,
		$_{i_1}{(b(i_1'))}_{i_1}$,
		$_{j_2}{(b(\xi'_{21}))}_{i_1}$,
		$_{j_2}{(b(\rho'_2,\xi'_1))}_{i_1}$,
		$_{i_0}{(b(i_0'))}_{i_0}$,
		$_{i_2}{(b(\xi'_1))}_{i_1}$,
		$_{i_1}{(b(\eta'_1))}_{i_0}$,
		$_{j_1}{(b(\xi'_3))}_{j_2}$,
		$_{j_0}{(b(\eta'_3,\xi'_{321}))}_{i_1}$,
		$_{j_1}{(b(\xi'_3,\xi'_{21},\eta'_1))}_{i_0}$,
		$_{j_0}{(b(\eta'_{32}))}_{i_1}$,
		$_{j_2}{(b(\rho'_2))}_{i_2}$,
		$_{j_0}{(b(\eta'_3,\eta'_{21}))}_{i_0}$,
		$_{j_0}{(b(\eta'_3,\eta'_2,\eta'_1))}_{i_0}$,
		$_{j_0}{(b(\eta'_3,\xi'_3,\rho'_2))}_{i_2}$,
		$_{j_1}{(b(\xi'_{321}))}_{i_1}$,
		$_{j_2}{(b(\xi'_2,\xi'_1))}_{i_1}$,
		$_{j_1}{(b(\xi'_3,\rho'_2,\xi'_1))}_{i_1}$,
		$_{j_0}{(b(\eta'_3,\xi'_3,\xi'_{21},\eta'_1))}_{i_0}$,
		$_{j_0}{(b(\eta'_{32},\eta'_1))}_{i_0}$,
		$_{j_1}{(b(\xi'_3,\rho'_2))}_{i_2}$,
		$_{j_0}{(b(\eta'_3,\xi'_{32},\xi'_1))}_{i_1}$,
		$_{j_1}{(b(\eta'_2))}_{i_1}$,
		$_{j_0}{(b(\eta'_3,\xi'_3,\xi'_2,\xi'_1))}_{i_1}$,
		$_{j_1}{(b(\eta'_2,\eta'_1))}_{i_0}$,
		$_{j_0}{(b(\eta'_3,\xi'_{321},\eta'_1))}_{i_0}$,
		$_{j_0}{(b(\eta'_3,\xi'_3))}_{j_2}$,
		$_{j_1}{(b(\xi'_{32},\xi'_1,\eta'_1))}_{i_0}$,
		$_{i_2}{(b(\xi'_1,\eta'_1))}_{i_0}$,
		$_{j_1}{(b(\xi'_3,\xi'_2,\xi'_1))}_{i_1}$.

		\vspace{0.2cm}

		Actions:

		$b(\xi'_{32},\xi'_1)\longrightarrow \eta_3 \otimes b(\eta'_3,\xi'_{32},\xi'_1) \otimes 1 $, 
		$b(\xi'_{32},\xi'_1)\longrightarrow 1 \otimes b(\xi'_{32},\xi'_1,\eta'_1) \otimes \eta_1 $, 
		$b(\xi'_{32},\xi'_1)\longrightarrow 1 \otimes b(\xi'_3,\xi'_2,\xi'_1) \otimes 1 $, 
		$b(\xi'_3,\xi'_2)\longrightarrow \eta_3 \otimes b(\eta'_3,\xi'_3,\xi'_2) \otimes 1 $, 
		$b(\xi'_3,\xi'_2)\longrightarrow 1 \otimes b(\xi'_3,\xi'_2,\xi'_1) \otimes \xi_1 $, 
		$b(\eta'_3,\xi'_3,\rho'_2,\xi'_1)\longrightarrow 1 \otimes b(\eta'_3,\xi'_3,\rho'_2,\xi'_1,\eta'_1) \otimes \eta_1 $, 
		$b(\eta'_3,\xi'_{32})\longrightarrow 1 \otimes b(\eta'_3,\xi'_3,\xi'_2) \otimes 1 $, 
		$b(\eta'_3,\xi'_{32})\longrightarrow 1 \otimes b(\eta'_3,\xi'_{32},\xi'_1) \otimes \xi_1 $, 
		$b(\eta'_3,\xi'_{32},\xi'_1,\eta'_1)\longrightarrow 1 \otimes b(\eta'_3,\xi'_3,\xi'_2,\xi'_1,\eta'_1) \otimes 1 $, 
		$b(\eta'_3,\xi'_3,\xi'_2)\longrightarrow 1 \otimes b(\eta'_3,\xi'_3,\xi'_2,\xi'_1) \otimes \xi_1 $, 
		$b(\eta'_{21})\longrightarrow \eta_3 \otimes b(\eta'_3,\eta'_{21}) \otimes 1 $, 
		$b(\eta'_{21})\longrightarrow 1 \otimes b(\eta'_2,\eta'_1) \otimes 1 $, 
		$b(\xi'_2)\longrightarrow \xi_3 \otimes b(\xi'_3,\xi'_2) \otimes 1 $, 
		$b(\xi'_2)\longrightarrow 1 \otimes b(\xi'_2,\xi'_1) \otimes \xi_1 $, 
		$b(\xi'_2,\xi'_1,\eta'_1)\longrightarrow \xi_3 \otimes b(\xi'_3,\xi'_2,\xi'_1,\eta'_1) \otimes 1 $, 
		$b(\xi'_{32})\longrightarrow 1 \otimes b(\xi'_{32},\xi'_1) \otimes \xi_1 $, 
		$b(\xi'_{32})\longrightarrow 1 \otimes b(\xi'_3,\xi'_2) \otimes 1 $, 
		$b(\xi'_{32})\longrightarrow \eta_3 \otimes b(\eta'_3,\xi'_{32}) \otimes 1 $, 
		$b(\xi'_{21},\eta'_1)\longrightarrow 1 \otimes b(\xi'_2,\xi'_1,\eta'_1) \otimes 1 $, 
		$b(\xi'_{21},\eta'_1)\longrightarrow \xi_3 \otimes b(\xi'_3,\xi'_{21},\eta'_1) \otimes 1 $, 
		$b(\eta'_{321})\longrightarrow 1 \otimes b(\eta'_3,\eta'_{21}) \otimes 1 $, 
		$b(\eta'_{321})\longrightarrow 1 \otimes b(\eta'_{32},\eta'_1) \otimes 1 $, 
		$b(\rho'_2,\xi'_1,\eta'_1)\longrightarrow \xi_3 \otimes b(\xi'_3,\rho'_2,\xi'_1,\eta'_1) \otimes 1 $, 
		$b(\eta'_3)\longrightarrow 1 \otimes b(\eta'_3,\xi'_{32}) \otimes \xi_{23} $, 
		$b(\eta'_3)\longrightarrow 1 \otimes b(\eta'_3,\eta'_2) \otimes \eta_2 $, 
		$b(\eta'_3)\longrightarrow 1 \otimes b(\eta'_3,\xi'_{321}) \otimes \xi_{123} $, 
		$b(\eta'_3)\longrightarrow 1 \otimes b(\eta'_3,\eta'_{21}) \otimes \eta_{12} $, 
		$b(\eta'_3)\longrightarrow 1 \otimes b(\eta'_3,\xi'_3) \otimes \xi_3 $, 
		$b(\xi'_3,\xi'_2,\xi'_1,\eta'_1)\longrightarrow \eta_3 \otimes b(\eta'_3,\xi'_3,\xi'_2,\xi'_1,\eta'_1) \otimes 1 $, 
		$b(\xi'_{321},\eta'_1)\longrightarrow 1 \otimes b(\xi'_3,\xi'_{21},\eta'_1) \otimes 1 $, 
		$b(\xi'_{321},\eta'_1)\longrightarrow \eta_3 \otimes b(\eta'_3,\xi'_{321},\eta'_1) \otimes 1 $, 
		$b(\xi'_{321},\eta'_1)\longrightarrow 1 \otimes b(\xi'_{32},\xi'_1,\eta'_1) \otimes 1 $, 
		$b(i_2')\longrightarrow \xi_2 \otimes b(\xi'_2) \otimes 1 $, 
		$b(i_2')\longrightarrow \xi_{23} \otimes b(\xi'_{32}) \otimes 1 $, 
		$b(i_2')\longrightarrow 1 \otimes b(\xi'_1) \otimes \xi_1 $, 
		$b(i_2')\longrightarrow \rho_2 \otimes b(\rho'_2) \otimes 1 $, 
		$b(\xi'_3,\xi'_{21})\longrightarrow \eta_3 \otimes b(\eta'_3,\xi'_3,\xi'_{21}) \otimes 1 $, 
		$b(\xi'_3,\xi'_{21})\longrightarrow 1 \otimes b(\xi'_3,\xi'_{21},\eta'_1) \otimes \eta_1 $, 
		$b(\xi'_3,\xi'_{21})\longrightarrow 1 \otimes b(\xi'_3,\xi'_2,\xi'_1) \otimes 1 $, 
		$b(j_1')\longrightarrow 1 \otimes b(\eta'_{21}) \otimes \eta_{12} $, 
		$b(j_1')\longrightarrow 1 \otimes b(\xi'_{32}) \otimes \xi_{23} $, 
		$b(j_1')\longrightarrow \eta_3 \otimes b(\eta'_3) \otimes 1 $, 
		$b(j_1')\longrightarrow 1 \otimes b(\xi'_3) \otimes \xi_3 $, 
		$b(j_1')\longrightarrow 1 \otimes b(\xi'_{321}) \otimes \xi_{123} $, 
		$b(j_1')\longrightarrow 1 \otimes b(\eta'_2) \otimes \eta_2 $, 
		$b(\eta'_3,\xi'_3,\xi'_{21})\longrightarrow 1 \otimes b(\eta'_3,\xi'_3,\xi'_{21},\eta'_1) \otimes \eta_1 $, 
		$b(\eta'_3,\xi'_3,\xi'_{21})\longrightarrow 1 \otimes b(\eta'_3,\xi'_3,\xi'_2,\xi'_1) \otimes 1 $, 
		$b(j_2')\longrightarrow 1 \otimes b(\xi'_2) \otimes \xi_2 $, 
		$b(j_2')\longrightarrow 1 \otimes b(\xi'_{21}) \otimes \xi_{12} $, 
		$b(j_2')\longrightarrow \xi_3 \otimes b(\xi'_3) \otimes 1 $, 
		$b(j_2')\longrightarrow 1 \otimes b(\rho'_2) \otimes \rho_2 $, 
		$b(j_0')\longrightarrow 1 \otimes b(\rho'_0) \otimes \rho_0 $, 
		$b(j_0')\longrightarrow 1 \otimes b(\eta'_{321}) \otimes \eta_{123} $, 
		$b(j_0')\longrightarrow 1 \otimes b(\eta'_3) \otimes \eta_3 $, 
		$b(j_0')\longrightarrow 1 \otimes b(\eta'_{32}) \otimes \eta_{23} $, 
		$b(\eta'_3,\eta'_2)\longrightarrow 1 \otimes b(\eta'_3,\eta'_2,\eta'_1) \otimes \eta_1 $,  
		$b(\xi'_3,\rho'_2,\xi'_1,\eta'_1)\longrightarrow \eta_3 \otimes b(\eta'_3,\xi'_3,\rho'_2,\xi'_1,\eta'_1) \otimes 1 $, 
		$b(i_1')\longrightarrow \xi_{12} \otimes b(\xi'_{21}) \otimes 1 $, 
		$b(i_1')\longrightarrow \xi_1 \otimes b(\xi'_1) \otimes 1 $, 
		$b(i_1')\longrightarrow 1 \otimes b(\eta'_1) \otimes \eta_1 $, 
		$b(i_1')\longrightarrow \eta_{23} \otimes b(\eta'_{32}) \otimes 1 $, 
		$b(i_1')\longrightarrow \xi_{123} \otimes b(\xi'_{321}) \otimes 1 $, 
		$b(i_1')\longrightarrow \eta_2 \otimes b(\eta'_2) \otimes 1 $, 
		$b(\xi'_{21})\longrightarrow 1 \otimes b(\xi'_{21},\eta'_1) \otimes \eta_1 $, 
		$b(\xi'_{21})\longrightarrow \xi_3 \otimes b(\xi'_3,\xi'_{21}) \otimes 1 $, 
		$b(\xi'_{21})\longrightarrow 1 \otimes b(\xi'_2,\xi'_1) \otimes 1 $, 
		$b(\rho'_2,\xi'_1)\longrightarrow 1 \otimes b(\rho'_2,\xi'_1,\eta'_1) \otimes \eta_1 $, 
		$b(\rho'_2,\xi'_1)\longrightarrow \xi_3 \otimes b(\xi'_3,\rho'_2,\xi'_1) \otimes 1 $, 
		$b(i_0')\longrightarrow \rho_0 \otimes b(\rho'_0) \otimes 1 $, 
		$b(i_0')\longrightarrow \eta_{12} \otimes b(\eta'_{21}) \otimes 1 $, 
		$b(i_0')\longrightarrow \eta_{123} \otimes b(\eta'_{321}) \otimes 1 $, 
		$b(i_0')\longrightarrow \eta_1 \otimes b(\eta'_1) \otimes 1 $, 
		$b(\xi'_1)\longrightarrow \xi_{23} \otimes b(\xi'_{32},\xi'_1) \otimes 1 $, 
		$b(\xi'_1)\longrightarrow \rho_2 \otimes b(\rho'_2,\xi'_1) \otimes 1 $, 
		$b(\xi'_1)\longrightarrow \xi_2 \otimes b(\xi'_2,\xi'_1) \otimes 1 $, 
		$b(\xi'_1)\longrightarrow 1 \otimes b(\xi'_1,\eta'_1) \otimes \eta_1 $, 
		$b(\eta'_1)\longrightarrow \xi_{12} \otimes b(\xi'_{21},\eta'_1) \otimes 1 $, 
		$b(\eta'_1)\longrightarrow \xi_{123} \otimes b(\xi'_{321},\eta'_1) \otimes 1 $, 
		$b(\eta'_1)\longrightarrow \eta_{23} \otimes b(\eta'_{32},\eta'_1) \otimes 1 $, 
		$b(\eta'_1)\longrightarrow \eta_2 \otimes b(\eta'_2,\eta'_1) \otimes 1 $, 
		$b(\eta'_1)\longrightarrow \xi_1 \otimes b(\xi'_1,\eta'_1) \otimes 1 $, 
		$b(\xi'_3)\longrightarrow 1 \otimes b(\xi'_3,\xi'_2) \otimes \xi_2 $, 
		$b(\xi'_3)\longrightarrow 1 \otimes b(\xi'_3,\xi'_{21}) \otimes \xi_{12} $, 
		$b(\xi'_3)\longrightarrow 1 \otimes b(\xi'_3,\rho'_2) \otimes \rho_2 $, 
		$b(\xi'_3)\longrightarrow \eta_3 \otimes b(\eta'_3,\xi'_3) \otimes 1 $, 
		$b(\eta'_3,\xi'_{321})\longrightarrow 1 \otimes b(\eta'_3,\xi'_3,\xi'_{21}) \otimes 1 $, 
		$b(\eta'_3,\xi'_{321})\longrightarrow 1 \otimes b(\eta'_3,\xi'_{32},\xi'_1) \otimes 1 $, 
		$b(\eta'_3,\xi'_{321})\longrightarrow 1 \otimes b(\eta'_3,\xi'_{321},\eta'_1) \otimes \eta_1 $, 
		$b(\xi'_3,\xi'_{21},\eta'_1)\longrightarrow 1 \otimes b(\xi'_3,\xi'_2,\xi'_1,\eta'_1) \otimes 1 $, 
		$b(\xi'_3,\xi'_{21},\eta'_1)\longrightarrow \eta_3 \otimes b(\eta'_3,\xi'_3,\xi'_{21},\eta'_1) \otimes 1 $, 
		$b(\eta'_{32})\longrightarrow 1 \otimes b(\eta'_3,\eta'_2) \otimes 1 $, 
		$b(\eta'_{32})\longrightarrow 1 \otimes b(\eta'_{32},\eta'_1) \otimes \eta_1 $, 
		$b(\rho'_2)\longrightarrow 1 \otimes b(\rho'_2,\xi'_1) \otimes \xi_1 $, 
		$b(\rho'_2)\longrightarrow \xi_3 \otimes b(\xi'_3,\rho'_2) \otimes 1 $, 
		$b(\eta'_3,\eta'_{21})\longrightarrow 1 \otimes b(\eta'_3,\eta'_2,\eta'_1) \otimes 1 $, 
		$b(\eta'_3,\xi'_3,\rho'_2)\longrightarrow 1 \otimes b(\eta'_3,\xi'_3,\rho'_2,\xi'_1) \otimes \xi_1 $, 
		$b(\xi'_{321})\longrightarrow 1 \otimes b(\xi'_{32},\xi'_1) \otimes 1 $, 
		$b(\xi'_{321})\longrightarrow 1 \otimes b(\xi'_{321},\eta'_1) \otimes \eta_1 $, 
		$b(\xi'_{321})\longrightarrow 1 \otimes b(\xi'_3,\xi'_{21}) \otimes 1 $, 
		$b(\xi'_{321})\longrightarrow \eta_3 \otimes b(\eta'_3,\xi'_{321}) \otimes 1 $, 
		$b(\xi'_2,\xi'_1)\longrightarrow 1 \otimes b(\xi'_2,\xi'_1,\eta'_1) \otimes \eta_1 $, 
		$b(\xi'_2,\xi'_1)\longrightarrow \xi_3 \otimes b(\xi'_3,\xi'_2,\xi'_1) \otimes 1 $, 
		$b(\xi'_3,\rho'_2,\xi'_1)\longrightarrow \eta_3 \otimes b(\eta'_3,\xi'_3,\rho'_2,\xi'_1) \otimes 1 $, 
		$b(\xi'_3,\rho'_2,\xi'_1)\longrightarrow 1 \otimes b(\xi'_3,\rho'_2,\xi'_1,\eta'_1) \otimes \eta_1 $, 
		$b(\eta'_3,\xi'_3,\xi'_{21},\eta'_1)\longrightarrow 1 \otimes b(\eta'_3,\xi'_3,\xi'_2,\xi'_1,\eta'_1) \otimes 1 $, 
		$b(\eta'_{32},\eta'_1)\longrightarrow 1 \otimes b(\eta'_3,\eta'_2,\eta'_1) \otimes 1 $, 
		$b(\xi'_3,\rho'_2)\longrightarrow \eta_3 \otimes b(\eta'_3,\xi'_3,\rho'_2) \otimes 1 $, 
		$b(\xi'_3,\rho'_2)\longrightarrow 1 \otimes b(\xi'_3,\rho'_2,\xi'_1) \otimes \xi_1 $, 
		$b(\eta'_3,\xi'_{32},\xi'_1)\longrightarrow 1 \otimes b(\eta'_3,\xi'_{32},\xi'_1,\eta'_1) \otimes \eta_1 $, 
		$b(\eta'_3,\xi'_{32},\xi'_1)\longrightarrow 1 \otimes b(\eta'_3,\xi'_3,\xi'_2,\xi'_1) \otimes 1 $, 
		$b(\eta'_2)\longrightarrow \eta_3 \otimes b(\eta'_3,\eta'_2) \otimes 1 $, 
		$b(\eta'_2)\longrightarrow 1 \otimes b(\eta'_2,\eta'_1) \otimes \eta_1 $, 
		$b(\eta'_3,\xi'_3,\xi'_2,\xi'_1)\longrightarrow 1 \otimes b(\eta'_3,\xi'_3,\xi'_2,\xi'_1,\eta'_1) \otimes \eta_1 $, 
		$b(\eta'_2,\eta'_1)\longrightarrow \eta_3 \otimes b(\eta'_3,\eta'_2,\eta'_1) \otimes 1 $, 
		$b(\eta'_3,\xi'_{321},\eta'_1)\longrightarrow 1 \otimes b(\eta'_3,\xi'_{32},\xi'_1,\eta'_1) \otimes 1 $, 
		$b(\eta'_3,\xi'_{321},\eta'_1)\longrightarrow 1 \otimes b(\eta'_3,\xi'_3,\xi'_{21},\eta'_1) \otimes 1 $, 
		$b(\eta'_3,\xi'_3)\longrightarrow 1 \otimes b(\eta'_3,\xi'_3,\xi'_2) \otimes \xi_2 $, 
		$b(\eta'_3,\xi'_3)\longrightarrow 1 \otimes b(\eta'_3,\xi'_3,\xi'_{21}) \otimes \xi_{12} $, 
		$b(\eta'_3,\xi'_3)\longrightarrow 1 \otimes b(\eta'_3,\xi'_3,\rho'_2) \otimes \rho_2 $, 
		$b(\xi'_{32},\xi'_1,\eta'_1)\longrightarrow \eta_3 \otimes b(\eta'_3,\xi'_{32},\xi'_1,\eta'_1) \otimes 1 $, 
		$b(\xi'_{32},\xi'_1,\eta'_1)\longrightarrow 1 \otimes b(\xi'_3,\xi'_2,\xi'_1,\eta'_1) \otimes 1 $, 
		$b(\xi'_1,\eta'_1)\longrightarrow \xi_2 \otimes b(\xi'_2,\xi'_1,\eta'_1) \otimes 1 $, 
		$b(\xi'_1,\eta'_1)\longrightarrow \rho_2 \otimes b(\rho'_2,\xi'_1,\eta'_1) \otimes 1 $, 
		$b(\xi'_1,\eta'_1)\longrightarrow \xi_{23} \otimes b(\xi'_{32},\xi'_1,\eta'_1) \otimes 1 $, 
		$b(\xi'_3,\xi'_2,\xi'_1)\longrightarrow 1 \otimes b(\xi'_3,\xi'_2,\xi'_1,\eta'_1) \otimes \eta_1 $, 
		$b(\xi'_3,\xi'_2,\xi'_1)\longrightarrow \eta_3 \otimes b(\eta'_3,\xi'_3,\xi'_2,\xi'_1) \otimes 1 $.  
		\end{any}
	}

\printbibliography

\end{document}